\documentclass{amsart}
\usepackage{fullpage}
\usepackage{amsmath, amssymb}
\usepackage{amsthm}
\usepackage{amscd}
\usepackage{color}
\usepackage{tikz}
\usepackage[all,cmtip]{xy}
\usepackage{xcolor}
\usepackage{pdfcomment}
\usepackage{bbm}
\usepackage[all,cmtip]{xy}
\usepackage{blkarray}
\usepackage{amsrefs}

\newtheorem{thm}{Theorem}[section]
\newtheorem{proposition}[thm]{Proposition}
\newtheorem{corollary}[thm]{Corollary}
\newtheorem{lemma}[thm]{Lemma}
\newtheorem{definition}[thm]{Definition}

\newtheorem{remark}[thm]{Remark}

\newtheorem{notation}{Notation}

\newcommand{\Ext}{\mathrm{Ext}}
\newcommand{\End}{\mathrm{End}}
\newcommand{\Hom}{\mathrm{Hom}}
\newcommand{\add}{\mathrm{add}}
\newcommand{\ind}{\mathrm{ind}}
\newcommand{\rad}{\mathrm{rad}}

\newcommand{\wt}{\widetilde}

\newcommand{\cN}{\mathcal{N}}

\newcommand{\La}{\Lambda}
\newcommand{\cB}{\mathcal{B}}
\newcommand{\cC}{\mathcal{C}}

\newcommand{\cU}{\mathcal{U}}

\newcommand{\bC}{\mathbb{C}}
\newcommand{\bE}{\mathbb{E}}
\newcommand{\bL}{\mathbb{L}}
\newcommand{\bA}{\mathbb{A}}
\newcommand{\bD}{\mathbb{D}}
\newcommand{\bG}{\mathbb{G}}
\newcommand{\bB}{\mathbb{B}}

\newcommand{\bZ}{\mathbb{Z}}

\newcommand{\bF}{\mathbb{F}}

\newcommand{\bK}{\mathbb{K}}
\newcommand{\Modu}{\mathrm{Mod}}
\newcommand{\modu}{\mathrm{mod}}

\begin{document}

\title[Cluster tilting modules for  mesh algebras]{Cluster tilting modules for  mesh algebras}
%\thanks{The third author would like to kindly acknowledge the financial support offered by Somerville College during her JRF}
\author{Karin Erdmann}
\address{Karin Erdmann, Mathematical Institute,
Andrew Wiles Building,
Radcliffe Observatory Quarter,
Woodstock Road,
Oxford OX2 6GG, United Kingdom
}
\email{erdmann@maths.ox.ac.uk}
\urladdr{http://people.maths.ox.ac.uk/erdmann/}
\author{Sira Gratz}
\address{
Sira Gratz, School of Mathematics and Statistics,
University of Glasgow,
University Place,
Glasgow G12 8QQ, United Kingdom
}
\email{Sira.Gratz@glasgow.ac.uk}
\urladdr{https://www.maths.gla.ac.uk/~sgratz/}
\author{Lisa Lamberti}
\address{
Lisa Lamberti, Department of Biosystems Science and Engineering, ETH Z\"urich, Basel, Switzerland; SIB Swiss Institute of Bioinformatics, Basel Switzerland
%Mattenstrasse 26,
%4058 Basel, Switzerland
}
\email{lisa.lamberti@bsse.ethz.ch}
%\urladdr{https://www.bsse.ethz.ch/cbg/group/people/person-detail.html?persid=158462}

\begin{abstract}
	\noindent We study cluster tilting modules in mesh algebras of Dynkin type, providing a new proof for their existence. Except for type $\bG_2$, we show that these are precisely the maximal rigid modules, and that %in all but one case we show 
	they are equivariant for a certain automorphism. 
	We further study their mutation, providing an example of mutation in an abelian category which is not stably 2-Calabi-Yau, and explicitly describe the combinatorics.
\end{abstract}

\maketitle

%\tableofcontents

\section{Introduction}

In recent years, cluster tilting theory has gained traction in the study of representation theory of finite dimensional algebras, and in algebraic Lie theory. On one hand it is a tool to study combinatorial phenomena arising in cluster theory, in the context of additive categorifications of cluster algebras. On the other hand, it generalizes classical tilting theory, which is crucial in the understanding of derived equivalences. 

In this article, we study 2-cluster tilting modules (cluster tilting modules for short, cf.\ Definition \ref{D:2-ct}) for finite dimensional self-injective algebras. The existence of a cluster tilting module for such an algebra $A$ has powerful implications. It was shown by Erdmann and Holm in \cite{EH} that if $A$ has cluster tilting modules, all modules must have complexity at most $1$, that is, the terms in a minimal projective resolution have bounded dimensions. Furthermore, the representation dimension of $A$ must be at most $3$. The notion of representation dimension was introduced by Auslander \cite{Auslander}, who also showed that the representation dimension of an algebra is at most $2$ if and only if it is of finite type. In this sense, the existence of a cluster tilting module shows that the algebra $A$ is not too far from being representation finite. 

Cluster tilting modules for self-injective algebras are notoriously elusive. In \cite{EH}, the only examples found were for certain algebras of finite type, and it is shown that block algebras of infinite representation type cannot have cluster tilting modules. However, there are cluster tilting modules for an important class of finite dimensional self-injective algebras: In a series of papers (including \cite{GLS-2}, \cite{GLS-3}, \cite{GLS-1}), Gei\ss, Leclerc and Schr\"oer showed the existence of cluster tilting modules for preprojective algebras of simply laced Dynkin type, and studied the cluster, as well as representation, theoretic implications. 

In this article, our main objects of interest are mesh algebras of Dynkin type, which were studied by Erdmann and Skowro\'{n}ski in \cite{ES}. Each mesh algebra is associated to a generalized Cartan matrix, with type either an arbitrary Dynkin type, or a type called $\bL_n$ (for $n \geq 1$). They are a natural generalization of preprojective algebras, which are precisely the mesh algebras of simply laced Dynkin type. Therefore, it is a natural question to ask whether mesh algebras have cluster tilting modules. 
Our first main result shows that all mesh algebras (except those of type $\bL_n$) indeed have cluster tilting modules (cf.\ Theorem \ref{thm:La_2_ct}).
The proof uses a result by Darp\"o and Iyama \cite{DI} (which they also use to find cluster tilting modules), and exploits the work of Gei\ss, Leclerc and Schr\"oer  \cite{GLS-3}, \cite{GLS-1}. After the completion of this paper, it has been brought to our attention that Asai \cite{Asai} has already shown (by a different construction) that a mesh algebra $\La$ has a cluster tilting module, which is equivariant under a certain autoequivalence, if and only if it is not of type $\bL_n$. In fact, coupling this with an observation by Yang, Zhang and Zhu \cite{YZZ}, it follows that the mesh algebra $\La$ has a cluster tilting module if and only if it is not of type $\bL_n$.

Hence, mesh algebras of (non-simply laced) Dynkin type provide a new example of a naturally occurring class of finite dimensional self-injective algebras that have cluster tilting modules. In particular, this yields that these algebras are 2-representation finite, in the sense of Iyama, see \cite{I}. Furthermore, we show that the cluster tilting modules coincide with the maximal rigid modules (cf.\ Theorem \ref{T:maximal rigid = 2-ct}). 

Our second main result describes mutation of cluster tilting modules (i.e.\ maximal rigid modules) for a mesh algebra $\La$ of Dynkin type other than $\bG_2$. Mutation replaces a summand of a cluster tilting module by a unique other module to obtain what is again a cluster tilting module, and this mutation is encoded in certain short exact sequences, called exchange sequences. Classically, in a (stably) 2-Calabi Yau setting, the summands we replace are indecomposable. In our case, however, they need not be. Instead, they are what we call minimal $\gamma$-equivariant (cf.\ Definition \ref{D:minimal equivariant}) under a certain automorphism $\gamma$. The number of non-isomorphic minimal $\gamma$-equivariant summands in a cluster tilting module is exactly the number of positive roots in the Dynkin diagram associated to $\La$ (cf.\ Proposition \ref{P:number of orbits}). This provides a rare instance of an explicit example of mutation on the level of a module category, extending work of Gei\ss, Leclerc and Schr\"oer \cite{GLS-3}.

We are able to exploit a symmetry of $\Ext^1$-spaces for modules which are equivariant under $\gamma$ (cf.\ Proposition \ref{P:Ext-symmetry}), to describe mutation of cluster tilting modules of $\La$. The key observation is that, if the automorphism $\gamma$ has order $\leq 2$, which is the case for all mesh algebras of Dynkin type except $\bG_2$, then every basic maximal rigid module is $\gamma$-equivariant. More generally, our mutation theory can be applied to any finite dimensional self-injective algebra for which a suitable automorphism $\gamma$ exists, as long as we restrict ourselves to cluster tilting modules that are $\gamma$-equivariant (provided those exist, as they do for example for $P(\bG_2)$). On the stable level, our mutation agrees with the mutation of $2$-cluster tilting subcategories with respect to almost complete $2$-cluster tilting subcategories as described by Iyama and Yoshino in \cite{IY}. However, the mutation we describe takes place already on the level of the module category, and is conscious of the twist by $\gamma$; in fact, we show that while mutation at a minimal $\gamma$-equivariant summand is involutive, the indecomposable summands of the $\gamma$-equivariant summand get shuffled. For an approach to combinatorially similar mutations, with different starting conditions, see Demonet's work \cite{Demonet}.

\smallskip

The paper is structured as follows. In Section \ref{S:Preliminaries} we recall some preliminaries on Galois covers, mesh algebras and cluster tilting modules. In Section \ref{S:Existence} we show the existence of cluster tilting modules for mesh algebras of Dynkin type. In Section \ref{S:Automorphisms} we discuss certain useful automorphisms of our algebras, and show that we can rely on $\Ext^1$-symmetries for modules which are equivariant under a certain automorphism. In Section \ref{S:Mutation} we introduce the notion of mutation of maximal rigid modules for a specific class of self-injective algebras, which includes the mesh algebras $P(\bB_k)$ for $k \geq 2$, $P(\bC_n)$ for $n \geq 3$ and $P(\bF_4)$. In Section \ref{S:Global dimension} we study the endomorphism algebras of basic maximal rigid modules for such an algebra $\La$. In Section \ref{S:rep dimension}, we show that when $\La$ is a mesh algebra of Dynkin type, the global dimension of these endomorphism rings is $3$, and hence the representation dimension of $\La$ is at most $3$. Finally, in Section \ref{S:matrix mutation}, we explicitly describe the combinatorics of the mutation of what we call admissible cluster tilting modules.

\section{Preliminaries} \label{S:Preliminaries}

Throughout this article, we work over an algebraically closed field $\bK$.

\subsection{Mesh categories}

Let $\Delta$ be an orientation of a simply laced Dynkin diagram, and let $\cC= \bK(\bZ\Delta)$ 
be the mesh category of the translation quiver
$(\bZ \Delta, \tau)$ (see Happel's work \cite[Chapter~I,~5.6]{Happel2} for a reminder on this construction). It is well-known and easy to see that the mesh category $\cC$
does not depend on the orientation of $\Delta$.
The vertices in $\bZ\Delta$ are labelled by $\bZ \times \Delta_0$, where $\Delta_0$ denotes the vertex set of $\Delta$, and arrows are given as follows: For every arrow $\alpha$ in $\Delta$ starting in $v \in \Delta_0$ and ending in $w \in \Delta_0$ and for every $i \in \bZ$ we get a pair of arrows
\begin{itemize}
	\item{$\alpha_i$ starting in $(i,v)$ and ending in $(i,w)$ and}
	\item{$\alpha'_i$ starting in $(i,w)$ and ending in $(i-1,v)$.}
\end{itemize}
Following the convention adopted in \cite[Section~9]{GLS-2} we denote the vertex in $\bZ\Delta$ labelled by $(i,q)$ by $q_i$. With this notation, 
the translation on $\mathbb{Z}\Delta$ is given by 
\[
	\tau(q_i) = q_{i+1}.
\]  

In Figure \ref{fig:standard labelling} we give an illustration of the 
translation quiver $(\bZ\Delta, \tau)$ for $\Delta$ of type $\bA_{2k-1}$, and with labelling induced by the labelling of the orientation of $\bA_{2k-1}$ as given in Figure \ref{fig:orientation}. 

\begin{figure}
 \begin{center}
\begin{tikzpicture}[scale = 0.6, font=\small]

%P(\bA_5)
%goes to infinity
\node (x) at (-1,7) {$\vdots$};
\node (x) at (1,7) {$\vdots$};
\node (x) at (-1,-1) {$\vdots$};
\node (x) at (1,-1) {$\vdots$};

%draw nodes
\node (32) at (-2,6) {$\ldots$};
\node (03) at (0,6) {$0_3$};
\node (42) at (2,6) {$\ldots$};
%\node (32) at (-4,6) {$\bullet$};
%\node (42) at (4,6) {$(2,4)$};

\node (12) at (-1,5) {$1_2$};
\node (22) at (1,5) {$2_2$};
\node (12k-3) at (-3,5) {$(2k-3)_1$};
\node (12k-2) at (3,5) {$(2k-2)_1$};

\node (31) at (-2,4) {$\ldots$};
\node (02) at (0,4) {$0_2$};
\node (41) at (2,4) {$\ldots$};

\node (11) at (-1,3) {$1_1$};
\node (21) at (1,3) {$2_1$};
\node (02k-3) at (-3,3) {$(2k-3)_0$};
\node (02k-2) at (3,3) {$(2k-2)_0$};

\node (30) at (-2,2) {$\ldots$};
\node (01) at (0,2) {$0_1$};
\node (40) at (2,2) {$\ldots$};

\node (10) at (-1,1) {$1_0$};
\node (20) at (1,1) {$2_0$};
\node (-12k-3) at (-3,1) {$(2k-3)_{(-1)}$};
\node (-12k-2) at (3,1) {$(2k-2)_{(-1)}$};

\node (00) at (0,0) {$0_0$};

\node (a) at (-2,0) {$\ldots$};
\node (b) at (2,0) {$\ldots$};

%draw arrows
\draw[->] (32) -- (12);
\draw[->] (03) -- (12);
\draw[->] (03) -- (22);
\draw[->] (42) -- (22);

\draw[->] (12) -- (31);
\draw[->] (12) -- (02);
\draw[->] (22) -- (02);
\draw[->] (22) -- (41);

\draw[->] (31) -- (11);
\draw[->] (02) -- (11);
\draw[->] (02) -- (21);
\draw[->] (41) -- (21);

\draw[->] (11) -- (30);
\draw[->] (11) -- (01);
\draw[->] (21) -- (01);
\draw[->] (21) -- (40);

\draw[->] (30) -- (10);
\draw[->] (01) -- (10);
\draw[->] (01) -- (20);
\draw[->] (40) -- (20);

\draw[->] (10) -- (00);
\draw[->] (20) -- (00);

\draw[->] (10) -- (a);
\draw[->] (20) -- (b);

\draw[->] (32) -- (12k-3);
\draw[->] (42) -- (12k-2);

\draw[->] (31) -- (02k-3);
\draw[->] (41) -- (02k-2);

\draw[->] (30) -- (-12k-3);
\draw[->] (40) -- (-12k-2);

\draw[->] (12k-3)--(31);
\draw[->] (12k-2)--(41);

\draw[->] (02k-3)--(30);
\draw[->] (02k-2)--(40);

\draw[->] (-12k-3)--(a);
\draw[->] (-12k-2)--(b);

\end{tikzpicture}
\caption{The quiver $\mathbb{Z}\Delta$ of the mesh category $\cC = \bK(\mathbb{Z}\Delta)$ with the standard labelling induced by the labelling of the orientation $\Delta$ of $\bA_{2k-1}$ from Figure \ref{fig:orientation}}\label{fig:standard labelling}
\end{center}
\end{figure}
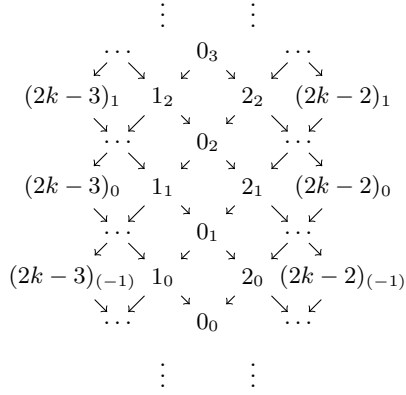

\begin{figure}
\begin{center}
\begin{tikzpicture}[scale=1.1,cap=round,>=latex]

%type A

%node style
%draw the vertices
\node[label=below:\footnotesize$2k-3$] (1) at (0,0) {\footnotesize$\bullet$};  
\node[label=below:\footnotesize$2k-5$] (2) at (1,0) {\footnotesize$\bullet$};  
\node (3) at (2,0) {$\ldots$};  
\node[label=below:\footnotesize$1$] (a) at (3,0) {\footnotesize$\bullet$};  
\node[label=below:\footnotesize$0$] (k) at (4,0) {\footnotesize$\bullet$};  
\node[label=below:\footnotesize$2$] (k+1) at (5,0) {\footnotesize$\bullet$};  
\node (b) at (6,0) {$\ldots$};  
\node[label=below:\footnotesize$2k-4$] (2k-2) at (7,0) {\footnotesize$\bullet$};  
\node[label=below:\footnotesize$2k-2$] (2k-1) at (8,0) {\footnotesize$\bullet$};    

%draw the arrows
\draw[->] (1) -- (2);
\draw[->] (2) -- (3);
\draw[->] (3) -- (a);
\draw[->] (a) -- (k);
\draw[<-] (k) -- (k+1);
\draw[<-] (k+1) -- (b);
\draw[<-] (b) -- (2k-2);
\draw[<-] (2k-2) -- (2k-1);

%type D

\begin{scope}[yshift = -40]
\node[label=below:\footnotesize$0$] (0) at (0,0.5) {\footnotesize$\bullet$};  
\node[label=below:\footnotesize$1$] (n) at (0,-0.5) {\footnotesize$\bullet$}; 
\node[label=below:\footnotesize$2$] (1) at (0.5,0) {\footnotesize$\bullet$};  
\node[label=below:\footnotesize$3$] (2) at (1.5,0) {\footnotesize$\bullet$};  
\node (3) at (2.5,0) {$\ldots$};  
\node[label=below:\footnotesize$n-1$] (a) at (3.5,0) {\footnotesize$\bullet$};  
\node[label=below:\footnotesize$n$] (k) at (4.5,0) {\footnotesize$\bullet$};

%draw the arrows
\draw[->] (0) -- (1);
\draw[->] (n) -- (1);
\draw[<-] (1) -- (2);
\draw[<-] (2) -- (3);
\draw[<-] (3) -- (a);
\draw[<-] (a) -- (k);

\end{scope}

%type E

\begin{scope}[yshift = -80]
\node[label=below:\footnotesize$4$] (1) at (0,0) {\footnotesize$\bullet$};  
\node[label=below:\footnotesize$2$] (2) at (1,0) {\footnotesize$\bullet$}; 
\node[label=above:\footnotesize$0$] (3) at (2,0) {\footnotesize$\bullet$}; 
\node[label=below:\footnotesize$1$] (a) at (3,0) {\footnotesize$\bullet$};  
\node[label=below:\footnotesize$3$] (k) at (4,0) {\footnotesize$\bullet$}; 
\node[label=below:\footnotesize$5$] (5) at (2,-1) {\footnotesize$\bullet$};

%draw the arrows
\draw[->] (1) -- (2);
\draw[->] (2) -- (3);
\draw[<-] (3) -- (a);
\draw[<-] (a) -- (k);
\draw[->] (3) -- (5);

\end{scope}

\end{tikzpicture}
\end{center}
\caption{Quivers of type $\bA_{2k-1}$, $\bD_{n+1}$ and $\bE_6$ with convenient symmetries} \label{fig:orientation}
\end{figure}

\subsection{Orbit categories of mesh categories}

The mesh category $\cC$ is a 
locally bounded $\bK$-linear category such that its additive closure $\add(\cC)$ is Krull-Schmidt,
and we can apply the results of Darp\"o and Iyama given in \cite{DI}. 
Here, locally bounded means that for all objects $x$ in $\cC$ we have
\[
	\sum_{y \in \cC} (\dim_{\bK} \Hom_\cC(x,y) + \dim_{\bK} \Hom_\cC(y,x)) < \infty.
\]

Assume $G$ is a group of $\bK$-linear automorphisms of the category $\cC$. 
Following \cite{DI}, the action of $G$ is admissible if $g(x) \not\cong x$ for
every object $x$ in $\cC$ and every $1\neq g\in G$. 
Assuming this,
the orbit category $\cC/G$ 
is again a locally bounded $\bK$-linear category such that its additive closure $\add(\cC/G)$ is Krull-Schmidt.  

We focus on the case where $G$ is induced by graph automorphisms of $\bZ\Delta$ which have
finitely many orbits on vertices of $\bZ\Delta$. In this case the action is admissible
if any $g \in G$ with $g\neq 1$ acts freely on $\bZ\Delta$. Assuming this, the orbit category
$\cC/G$ is the $\bK$-category of a finite-dimensional $\bK$-algebra $\bK(\bZ\Delta/G)$, whose
quiver has vertices and arrows labelled by the orbits of $G$, and with relations
induced by the mesh relations of $\cC$.

The algebra obtained in this way is called the {\em mesh algebra} of
$\bZ\Delta/G$, denoted by $\bK(\bZ\Delta/G)$. 
We will then, by common convention, identify the $\bK$-algebra $\bK(\bZ\Delta/G)$ with 
its $\bK$-category $\cC/G$.

\subsection{Mesh algebras}\label{S:list}

We consider the following automorphisms of $\cC$ induced by graph automorphisms of $\bZ\Delta$.

\begin{itemize}

\item[(i)]{First, the translation $\tau$ of $\cC$, which is 
admissible.}

\item[(ii)]{Second, suppose $\sigma$ is a graph automorphism of the underlying Dynkin diagram of $\Delta$, and suppose that the orientation of $\Delta$ is invariant under $\sigma$. Then $\sigma$ induces the following graph automorphism of $\bZ\Delta$ of finite order:
\[
	q_i \mapsto \sigma(q)_i, \; \; a_i \mapsto \sigma(a)_i,
\]
for all $i \in \bZ$, and every vertex $q$ and arrow $a$ of $\Delta$. By abuse of notation, we denote this graph automorphism of $\bZ\Delta$, and the induced automorphism of $\cC$, also by $\sigma$. The automorphism $\sigma$ commutes with $\tau$, and we consider the admissible automorphism 
$\sigma \tau$. 
}
\end{itemize}

This leads us to consider the following cases. We refer the reader to \cite{ES} for a detailed description of the respective mesh algebras.

\begin{enumerate}

\item{Let $G= \langle \tau \rangle$. Then $\cC/G \cong P(\Delta)$, the 
preprojective algebra of type $\Delta$. Recall that $P(\Delta)$ is defined as the algebra 
$K\overline{\Delta}/\langle I_\rho \rangle$, 
where $\overline{\Delta}$ is the double quiver of $\Delta$
obtained from $\Delta$ by adding an arrow $a^\ast$ starting in vertex $w$ and ending in vertex $v$ for each arrow $a$ in $\Delta$ starting in $v$ and ending in $w$, and where $\langle I_\rho \rangle$ 
is the ideal generated by the element
$$
I_{\rho}=\sum_{a\in Q}(a^\ast a- aa^\ast).
$$
Note that $P(\Delta)$ only depends on the underlying Dynkin diagram of $\Delta$, and not on its orientation.
}

\item{ For $k \geq 2$ let $\Delta$ be of type $\bA_{2k-1}$, and consider the automorphism $\sigma$ of $\cC$ induced by the graph automorphism of order 2 of $\bA_{2k-1}$. Assume the orientation of $\Delta$ is invariant under $\sigma$, e.g.\ take the orientation of $\bA_{2k-1}$ from Figure \ref{fig:orientation}.
Then $\cC/\langle \sigma \tau \rangle$ is isomorphic to the mesh algebra $P(\bB_k)$.
}

\item{ For $n \geq 3$, let $\Delta$ be of type $\bD_{n+1}$, and consider the automorphism $\sigma$ on $\cC$ induced by the graph automorphism of order $2$ (or, if $n=3$, some choice $\sigma$ of automorphism of order $2$) of $\bD_{n+1}$. Assume the orientation of $\Delta$ is invariant under $\sigma$, e.g.\ take the orientation of $\bD_{n+1}$ from Figure \ref{fig:orientation}.
Then  $\cC/ \langle \sigma \tau \rangle$ is isomorphic to the mesh algebra $P(\bC_n)$. 
}

\item{Let $\Delta$ be of type $\bE_6$ and consider the automorphism $\sigma$ of $\cC$ induced by the graph automorphism
of order $2$ of $\bE_6$. Assume the orientation of $\Delta$ is invariant under $\sigma$, e.g.\ take the orientation of $\bE_{6}$ from Figure \ref{fig:orientation}.
Then $\cC/\langle \sigma \tau \rangle$ is isomorphic to the mesh algebra $P(\bF_4)$. 
}

\item{Let $\Delta$ be of type $\bD_4$, and consider the automorphism $\sigma$ of $\cC$ induced by the graph automorphism of order $3$ of $\bD_4$. Assume the orientation of $\Delta$ is invariant under $\sigma$, e.g.\ take the orientation of $\bD_{4}$ from Figure \ref{fig:orientation} (setting $n = 3$).
Then $\cC/ \langle \sigma \tau \rangle$ is isomorphic to the mesh algebra $P(\bG_2)$.
}
 \end{enumerate}

\subsection{Modules of $\cC$}

In line with \cite{GLS-2}, we work with left modules throughout.
For a $\bK$-linear category $\cB$ a  {\em left $\cB$-module} is a $\bK$-linear functor
from $\cB$ to the category of $\bK$-vector spaces. 
We denote by 
$\Modu \cB$ the category of left
$\cB$-modules, and by 
$\modu \cB$ the category of finitely presented left $\cB$-modules.

Assume that $G$ is a group of $\bK$-linear automorphisms of $\cB$, acting from the left on
$\cB$. Then $G$ acts naturally on $\Modu \cB$:
If $g\in G$ and $M$ is a left $\cB$-module, we
write 
\[
	g_*(M):= M(g^{-1}(-))
\] 
 (denoted by $M^g$ in \cite{GLS-2}). Here we use the notation from \cite{DI}, to facilitate comparing the relevant results we are referring to from this paper, even though, following \cite{GLS-2}, we do work with left modules throughout. If $\cU$ is a full subcategory of $\cB$ we write $g_*(\cU)$ for the full subcategory of objects $g_*(M)$ for $M$ in $\cU$. We say that a full subcategory $\cU$ of $\Modu \cB$ is {\em $G$-equivariant} if $g_*(\cU) = \cU$ for all $g\in G$.

\subsection{The ``start module''}\label{S:start module}

Let $\Delta$ be an orientation of a simply laced Dynkin diagram. Following  the approach in \cite[Section~2.4]{GLS-1}
we define the Auslander category $\Gamma_\Delta$ of $\Delta$  to be the full subcategory of $\cC$
whose objects are the vertices in the Auslander-Reiten quiver $A_\Delta$ of $\bK \Delta^{op}$, viewed as a subquiver of $\bZ \Delta$. 
%For example if $\Delta$ is of type $\bA_5$ with orientation as in Figure \ref{fig:orientation} for $k=3$, then $A_\Delta$ is depicted in Figure \ref{fig:AQ}. 
Note that $\modu \Gamma_\Delta$ can be identified with a full subcategory of $\modu \cC$.

For a group $G$ of $\bK$-linear automorphisms of $\cC$, the covering functor
$F \colon \cC \to \cC/G$  induces the pull-up 
\[
	F^* \colon \Modu \cC/G \to \Modu \cC
\] 
and the push-down (denoted by $F_{\lambda}$ by Bongartz and Gabriel in \cite{BG})
\[
	F_{*} \colon \Modu \cC \to \Modu \cC/G.
\]
Both of these functors are exact and $(F_*, F^*)$ is an adjoint pair. 
Note that restriction of $F_*$ to the subcategory of finitely presented modules induces
a functor $F_* \colon \modu \cC \to \modu \cC/G$. 

In \cite{GLS-1} the main object of study is the $P(\Delta)$-module $I_\Delta$, where, as in Section \ref{S:list}, $P(\Delta)$ denotes the preprojective algebra of type $\Delta$. Recall that $P(\Delta) \cong \cC / \langle \tau \rangle$, where $\cC = \bK (\bZ \Delta)$ with translation $\tau$ on $\bZ \Delta$, and set $F = \colon \cC \to \cC/\langle \tau \rangle$ to be the covering functor. For each vertex $x$ in the quiver $A_\Delta$ (we write $x \in A_\Delta$ for brevity), we denote by $S(x)$ the simple at $x$ and by $I(x)$ the injective $\Gamma_\Delta$-module
with socle $S(x)$. Observe that $I(x)$ is spanned by all paths in $A_\Delta$ ending at vertex $x$. Then the module $I_\Delta$ is the push-down 
\[
	I_\Delta = F_*(\bigoplus_{x\in A_\Delta} I(x)),
\]
of the $\cC$-module $\bigoplus_{x\in A_\Delta} I(x)$ which belongs to $\modu \Gamma_\Delta$, i.e.\ which
is supported on $A_\Delta$, to the preprojective algebra $P(\Delta)$. As in \cite{GLS-1}, we will refer to $I_\Delta$ as the {\em start module of $P(\Delta)$ with respect to $\Delta$}. Note that while $P(\Delta)$ does not depend on the orientation of $\Delta$, the start module $I_\Delta$ does.

\subsection{Cluster tilting modules}\label{subsect:2tilt}
Let $\mathcal{T}$ be a triangulated or abelian category. If $U$ is an object in $\mathcal{T}$, we write $\add(U)$ for its additive hull. An object $U$ in $\mathcal{T}$ is called {\em rigid}, if $\Ext_{\mathcal{T}}^1(U,U) = 0$. It is called {\em maximal rigid}, if whenever $X \oplus U$ is rigid, then $X \in \add(U)$.

\begin{definition}\label{D:2-ct}
Let $\mathcal{T}$ be a triangulated or abelian category and let $\mathcal{U}$ be a full subcategory of $\mathcal{T}$ that is closed under isomorphisms, finite direct sums and direct summands. Then $\mathcal{U}$ is a {\em 2-cluster tilting subcategory of $\mathcal{T}$}, or {\em cluster tilting subcategory for short} if it is functorially finite and it satisfies
\begin{eqnarray*}
	\mathcal{U} 	&=& \{T \in \mathcal{T} \mid \mathrm{Ext^1}(U,T)=0 \; \text{for all} \; U \in \mathcal{U}\}\\
				&=& \{T \in \mathcal{T} \mid \mathrm{Ext^1}(T,U)=0 \; \text{for all} \; U \in \mathcal{U}\}.
\end{eqnarray*}

If $\mathcal{U} = \mathrm{add}(U)$ for an object $U \in \mathcal{T}$, then we call $U$ a {\em 2-cluster tilting object of $\mathcal{T}$}, and if $\mathcal{T}$ is a category of modules, then we call $U$ a {\em 2-cluster tilting module}, or {\em cluster tilting module} for short.
\end{definition}

Note that any 2-cluster tilting object is maximal rigid, but the converse is not true in general. The next theorem provides a concrete example of a cluster tilting module, which will be essential for us.

\begin{thm}[{\cite[Theorem~1]{GLS-1},\cite[Theorem~2.2]{GLS-3}} ]\label{T:GLS}
	The start module $I_\Delta$ of $P(\Delta)$ with respect to $\Delta$ is a cluster tilting module of $P(\Delta)$.
\end{thm}

The following crucial result due to Darp\"o and Iyama allows us to exploit Theorem \ref{T:GLS} to show existence of cluster tilting modules for mesh algebras. While the result in \cite{DI} is stated more generally for $d$-cluster tilting subcategories and $d$-cluster tilting modules, we only use the case $d = 2$. Recall that our field $\bK$ is algebraically closed, and note that a $\bK$-linear category is Morita equivalent to its additive closure.

\begin{thm}[{\cite[Corollary~2.14]{DI}}]\label{C:move} Let $\mathcal{T}$ be a locally bounded $\bK$-linear category such that $\add(\mathcal{T})$ is Krull-Schmidt
and let $G$ be a finitely generated free abelian group, acting admissibly on 
$\mathcal{T}$. The push-down functor 
$F_*: \modu \mathcal{T} \to \modu \mathcal{T}/G$ induces a bijection between
\begin{itemize}
\item[(a)] the set of locally bounded $G$-equivariant cluster tilting subcategories of $\modu \mathcal{T}$;
\item[(b)] the set of locally bounded cluster tilting subcategories of $\modu \mathcal{T}/G$.
\end{itemize}
\end{thm}

\begin{remark}\label{R:move} \normalfont
In particular, it follows that in the set-up as in Theorem \ref{C:move}, the push-down $F_*$ induces a bijection between
\begin{itemize}
\item[(a)] the set of locally bounded $G$-equivariant cluster tilting subcategories of $\modu \mathcal{T}$ that have finitely many $G$-orbits of indecomposable objects;
\item[(b)] the set of basic cluster tilting modules in $\modu \mathcal{T}/G$.
\end{itemize}
\end{remark}

\section{Existence of cluster tilting modules for mesh algebras}\label{S:Existence}

In this section we show that all algebras from the list in Section \ref{S:list} have cluster tilting modules. By Theorem \ref{T:GLS} this is true for preprojective algebras of type $\Delta$, and in Theorem \ref{thm:La_2_ct} we show that it also holds for mesh algebras of non-simply laced Dynkin type.

\subsection{Invariance under twists}

Let $\Delta$ be an orientation of a simply laced Dynkin diagram, and consider a group $G$ of automorphisms acting admissibly on the mesh category $\cC=\bK(\bZ\Delta)$. 
Let $A_\Delta$ be the quiver of $\Gamma_\Delta$ as in Section \ref{S:start module}. For $g \in G$ denote by $g(\Gamma_\Delta)$ the full subcategory of $\cC$ with indecomposable objects 
\[
	\{g(x) \mid x \text{ is indecomposable in } \Gamma_\Delta\}.
\]
Its isomorphism classes of indecomposable objects are the vertices in the full subquiver $g(A_\Delta)$ of $\bZ \Delta$ with vertices $\{ g(x) \mid x \in A_\Delta\}$, where $g$ is considered as the induced graph automorphism of $\bZ \Delta$. Recall from Section \ref{S:start module} that $I(x)$ denotes the $\cC$-module with support on $A_\Delta$, which, as a $\Gamma_\Delta$-module, is injective with socle $S(x)$ for $x$ a vertex in $A_\Delta$. 

\begin{lemma} Let $g \in G$. The module $g_*I(x)$ is the $\cC$-module supported on
$g(A_\Delta)$ which is injective as a module for $g(\Gamma_\Delta)$ with socle
$S(g(x))$.
\end{lemma}

\begin{proof}
If $M$ is a $\cC$-module supported on $A_\Delta$ then $g_*(M) = M(g^{-1}(-))$ is
supported on $g(A_\Delta)$. (This can be thought of as $M$ ``shifted to $g(A_\Delta)$''.)
If $M$ has a simple socle $S(x)$ then  $g_*(M)$ has simple socle
$S(g(x))$. If $M$ is injective
as a module for $\Gamma_\Delta$ then the shift of $M$ by $g$ is injective as a
module for $g(\Gamma_\Delta)$.
\end{proof}

\begin{definition}
	Let $\cB$ be a $\bK$-linear category, and let $g$ be an automorphism of $\cB$. A module $M$ in $\modu \cB$ is called {\em $g$-equivariant}, if $g_*(M) \cong M$.
\end{definition}

\begin{corollary}\label{C:invariant} Assume $\sigma$ is an
automorphism of $\cC$ induced by a quiver automorphism of $\Delta$. Then the module $\bigoplus_{x\in A_\Delta} I(x)$ is $\sigma$-equivariant.
\end{corollary}

\subsection{Existence of cluster tilting modules}

\begin{thm}\label{thm:La_2_ct}
	Let $\La$ be a mesh algebra of non-simply laced Dynkin type, i.e. one of the algebras
		\begin{enumerate}
			\item{$P(\bB_k)$ for $k \geq 2$},
			\item{$P(\bC_n)$ for $n \geq 3$},
			\item{$P(\bG_2)$},
			\item{$P(\bF_4)$}.
		\end{enumerate}
	Then $\La$ has cluster tilting modules.
\end{thm}

\begin{proof}
	In each of the cases respectively, let $\Delta$ be the quiver and $\sigma$ the automorphism of $\cC$ induced by the graph automorphism, also denoted by $\sigma$, on the underlying diagram of $\Delta$ listed below:
	\begin{enumerate}
		\item{For  $k \geq 2$ and $\La = P(\bB_k)$ let $\Delta$ be the orientation of $\bA_{2k-1}$ from Figure \ref{fig:orientation} and let $\sigma$ be the graph automorphism of $\bA_{2k-1}$ of order two, i.e.\ the automorphism given by the permutation of vertices $(1,2)(3,4)\cdots(2k-3,2k-2)$.}
		\item{For $n \geq 3$ and $\La = P(\bC_n)$ let $\Delta$ be the orientation of $\bD_{n+1}$ from Figure \ref{fig:orientation}
 		and let $\sigma$ be the graph automorphism of $\bD_{n+1}$ of order two, i.e.\ the automorphism given by the permutation of vertices $(0,1)$.}
				\item{For $\La = P(\bG_2)$ let $\Delta$ be the orientation of $\bD_4$ from Figure \ref{fig:orientation} (setting $n = 3$)
		and let $\sigma$ be the graph automorphism of $\bD_4$ of order three, i.e.\ the automorphism given by the permutation of vertices $(0,1,3)$.}
		\item{For $\La = P(\bF_4)$ let $\Delta$ be the orientation of $\bE_6$ from Figure \ref{fig:orientation}
		and let $\sigma$ be the graph automorphism of $\bE_6$ of order two, i.e.\ the automorphism given by the permutation of vertices $(1,2)(3,4)$.}
	\end{enumerate}
	Consider now in each case the preprojective algebra $P(\Delta)$. By Theorem \ref{T:GLS}, the start module $I_\Delta$ of $P(\Delta)$ with respect to $\Delta$ is cluster tilting. Consider the push-down 
	\[
		F_* \colon \modu \cC \to \modu P(\Delta)
	\] 
of the covering functor. By Corollary \ref{C:move} and Remark \ref{R:move} the subcategory $F_*^{-1}(I_\Delta)$ is cluster tilting in $\modu \cC$, it is $\langle \tau \rangle$-equivariant and has finitely many $\langle \tau \rangle$-orbits of indecomposable objects. Since $\Delta$ is invariant under $\sigma$, by Corollary \ref{C:invariant} the $\cC$-module $\oplus_{x\in A_\Delta} I(x)$ is $\sigma$-equivariant 
and hence so is 
\[
	F_*^{-1}(I_\Delta) = F_*^{-1}(F_*(\oplus_{x\in A_\Delta} I(x))).
\] 
Therefore, the subcategory $F_*^{-1}(I_\Delta)$ is $\langle \sigma \tau \rangle$-equivariant, and there are finitely many $\langle \sigma \tau \rangle$-orbits of indecomposable objects.
Consider now the push-down
	\[
		\tilde{F}_* \colon \modu \cC \to \modu \cC/\langle \sigma \tau \rangle \cong \modu \La.
	\]
	Again by Theorem \ref{C:move} and Remark \ref{R:move}, the push-down $I_\La = \tilde{F}_*(\oplus_{x\in A_\Delta} I(x))$ is a cluster tilting module of $\La$, which proves the claim.
	
\end{proof}

\begin{remark}
	Note that with the proof of Theorem \ref{thm:La_2_ct} we provide a new existence proof, which is different to, and shorter than, the proof given by Asai in \cite{Asai}, which we only discovered after the completion of this paper. However, \cite{Asai} states additionally that the mesh algebra $P(\bL_n)$ does not have a cluster tilting module that is stably invariant under $\Sigma^{-2}\circ S$, where $\Sigma$ is the suspension and $S$ is the Serre functor on the stable category $\underline{\modu}P(\bL_n)$. We observe that, using \cite[Proposition~3.4]{IY}, it follows directly that a mesh algebra $\La$ has a cluster tilting module if and only if it is of Dynkin type.
\end{remark}

\section{Automorphisms and $\Ext^1$-symmetry for mesh algebras}\label{S:Automorphisms}

Throughout this section we denote by $\La$ a mesh algebra of non-simply laced Dynkin type, i.e.\ as in Theorem \ref{thm:La_2_ct}. 
Further, let $\Delta$ be the respective Dynkin quiver and $\sigma$ the respective automorphism of $\cC = \bK(\bZ\Delta)$, such that $\La = \cC/\langle \sigma \tau \rangle$, as outlined in the list from Section \ref{S:list}.

We recall and further investigate certain automorphisms of $\Lambda$, and show that for modules which are equivariant under a specific automorphism, we obtain a useful $\Ext^1$-symmetry. We first start with some important facts about general self-injective algebras.

\subsection{Automorphisms of a self-injective algebra} \label{S:Nakayama automorphism}

Let $A$ be a finite dimensional self-injective algebra. For an algebra automorphism $\varphi$ and a module $M$ in $\modu A$, we denote from now on by ${}_\varphi M$, instead of by $\varphi_*(M)$, the twist of $M$ by $\varphi$; since we now take the point of view of finite dimensional algebras and their naturally arising automorphisms, we switch to this more algebraic notation. Let $\cN$ be the Nakayama functor on $\modu A$, that is
\[
	\cN = D\Hom_A(-,A),
\]
where $D(-) = \Hom_\bK(-,\bK)$. It is well known that $D(A)$ is isomorphic to the twisted module
\[
	D(A) \cong {}_{\eta^{-1}}A_{\mathrm{id}},
\]
where $\eta$ is an algebra automorphism, which is constructed by Yamagata in \cite{Ya} using a non-degenerate associative bilinear form. The morphism $\eta$ is called a {\em Nakayama automorphism of $A$}. Note that for any module $Y$ in $\modu A$ we have
\begin{eqnarray*}
	\cN Y = D\Hom_A(Y,A) \cong DA \otimes_A Y \cong {}_{\eta^{-1}}A \otimes_A Y \cong {}_{\eta^{-1}}Y.
\end{eqnarray*}

We denote by $A^e$ the enveloping algebra of $A$ and by $\Omega$ the first syzygy in a bi-module resolution; note that left $A^e$ modules are just $A$-$A$-bimodules. Putting the next lemma into a wider context, note that if some syzygy of $A$ as an $A^e$-module is isomorphic to a twist of $A$ as a bimodule, then it follows that all left $A$-modules have complexity $\leq 1$: The terms of a minimal bimodule resolution of $A$ have bounded dimension, and tensoring this with a left module $Y$ yields a projective resolution of $Y$, where the terms still have bounded dimension. Therefore, the condition in \cite{EH} for the existence of cluster tilting modules is satisfied, and potentially $A$ might have cluster tilting modules. Here, we focus on algebras where $\Omega^3_{A^e}(A)$ is a twist of $A$ as a bimodule.

\begin{lemma}\label{L:Ext twist}
	Assume $A$ is a finite dimensional self-injective algebra such that 
	\[
		\Omega^3_{A^e}(A) \cong {}_\mu A_{\mathrm{id}}
	\]
	for some automorphism $\mu$ of $A$ and let $\eta$ be a Nakayama automorphism of $A$. 
	Consider the automorphism $\gamma = \eta^{-1} \circ \mu$
	on $A$, and let $X$ and $Y$ be modules in $\modu A$. Then
	\[
		D\Ext^1_A(Y,X) \cong \Ext^1_A(X, {}_\gamma Y).
	\]
\end{lemma}

\begin{proof}
The first assumption implies that for any left module $M$ of $A$ we have $\Omega^3(M) \cong {}_\mu M$.
	We underline whenever we work in the stable category $\underline{\modu}A$. Recall that suspension in $\underline{\modu}A$ is given by $\Omega^{-1}$ and that we have the Serre functor $\Omega \cN$. 
	Since $\cN$ is a self-Morita equivalence for $A$ (cf.\ for example Zimmermann's book \cite[Lemma~4.5.6]{Zimmermann}), it commutes with $\Omega$. So we obtain
	\begin{eqnarray*}
		D\Ext^1_A(Y,X) 	&\cong& D\underline{\Hom}(Y,\Omega^{-1}X)
					\cong \underline{\Hom}(\Omega^{-1}X,\Omega \cN Y)
					\cong  \underline{\Hom}(\Omega X,\Omega^3 \cN Y)
					 \cong \underline{\Hom}(\Omega X, \cN \Omega^3 Y) \\
					 &\cong&  \underline{\Hom}(\Omega X,\cN ({}_\mu Y))
					 \cong  \underline{\Hom}(\Omega X, {}_{\eta^{-1} \circ \mu}Y)
					 \cong \underline{\Hom}( X, \Omega^{-1} ({}_{\eta^{-1} \circ \mu}Y))
					 \cong\Ext^1_A (X, {}_{\gamma}Y).
	\end{eqnarray*}
\end{proof}

Consider now $\La$, our mesh algebra of Dynkin type. We know that $\Omega^3_{\Lambda^e}(\Lambda) \cong {}_\mu \Lambda_{\mathrm{id}}$ for some automorphism $\mu$ (cf.\ Section \ref{S:third syzygy}), so we will be able to apply Lemma \ref{L:Ext twist} to $\La$. In \cite{AS}, Andreu Juan and Saor\'{i}n explicitly describe the following automorphisms of $\Lambda$.
\begin{itemize}
	\item{The Nakayama automorphism $\eta$, which is induced by an automorphism $\theta$ of $\cC$.
	}
	\item{An algebra automorphism $\mu$ of $\Lambda$ such that $\Omega^3_{\La^e}(\La) \cong {}_{\mu}\La_1$. 
In particular they determine the period of $\La$ as a bimodule, and hence the order of $\mu$ in the factor group modulo inner automorphisms.}
\end{itemize}
We fix $\eta$ and $\mu$ to be these automorphisms for $\La$, and, for convenience of notation, discuss the latter in slightly more detail in 
Section \ref{S:third syzygy}. 

\subsection{The third syzygy}\label{S:third syzygy}

Consider the automorphism $\vartheta$ on $\cC$ which fixes each vertex of $\bZ \Delta$ and acts on arrows by
\[
	\vartheta(a) = (-1)^{s(\tau^{-1}(a)) + s(a)}a,
\]
where $s$ denotes the signature map described in \cite[Proposition~3.3]{AS}. To be more precise, first \cite{AS} define a set $X$ of 
arrows of $\cC$ which is invariant under the group $G$,
so that every mesh contains precisely one arrow in $X$.
Then the signature map $s$ is defined from the set of arrows of $\cC$ to $\{1, 0\}$ by 
$s(a) = 1$ for $a\in X$ and $s(a)=0$ otherwise. 

If $\Delta$ is of type other than $\bA_{2k-1}$, then $\vartheta$ acts as the identity (cf.\ \cite[Remark~5.4]{AS}). 
The ungraded version of \cite[Corollary~5.5]{AS} states the following:

\begin{proposition}[{\cite[Corollary~5.5]{AS}}]\label{P:Corollary AS}
The automorphism $\mu$ of $\Lambda$, such that $\Omega^3_{\La^e}(\La) \cong {}_{\mu}\La_1$, is induced by the automorphism
\begin{enumerate}
	\item $\theta \circ \tau^{-1} \circ \vartheta$ if $\Delta$ is of type $\bA_{2k-1}$, and
	\item $\theta \circ \tau^{-1}$ otherwise,
\end{enumerate}
of $\cC$, where $\theta$ is the automorphism of $\cC$ as in \cite[Corollary~5.5]{AS} inducing the Nakayama automorphism $\eta$ of $\Lambda$.
\end{proposition}

\subsection{The automorphism $\gamma$}\label{S:gamma}

Consider our mesh algebra $\La$ of Dynkin type with automorphisms $\mu$ as discussed in Section \ref{S:third syzygy} and Nakayama automorphism $\eta$, and as in Lemma \ref{L:Ext twist} set
\[
	\gamma = \eta^{-1} \circ \mu.
\]

Recall that $\La = \cC/G$ where $\cC = K(\bZ\Delta)$ is the mesh category and 
$G = \langle \sigma\tau \rangle$, for $\sigma$ and $\Delta$ as in the list from Section \ref{S:list}. The automorphisms $\sigma$ and $\tau$ commute
with each element of $G$ and hence induce  the automorphisms
$\wt{\sigma}$ and $\wt{\tau}$ of $\La$. 
Moreover, since $\wt{\sigma}\wt{\tau}$ is the identity for $\cC/G$ it follows that $\wt{\sigma} = \wt{\tau}^{-1}$.

\begin{lemma}\label{L:order 2}
The automorphism $\gamma$ of $\La$ is equal to $\wt{\sigma}$, up to inner automorphism.
\end{lemma}

\begin{proof}
Recall that $\mu$ is induced by the automorphism of $\cC$ from Proposition \ref{P:Corollary AS}. In case (2) of Proposition \ref{P:Corollary AS}, when $\Delta$ is not of type $\bA_{2k-1}$, we obtain 
\[
	\gamma = \eta^{-1} \circ \mu = \eta^{-1} \circ \eta\circ \wt{\tau}^{-1} = \wt{\tau}^{-1} = \wt{\sigma}.
\] 

Now consider case (1), where $\sigma$ has order 2. 
Recall the definition of $\vartheta$ from the beginning of Section \ref{S:third syzygy}, via the signature map $s$.
We have that $\tau \vartheta(a) =  -\tau(a)$ 
and $\vartheta(\tau(a)) = -\tau(a)$ for an arrow $a$. So $\tau$ and $\vartheta$ commute.

Now by definition, the signature map $s$ is constant on $G$-orbits and commutes with $\sigma\tau$. Also, as maps on arrows, $s\circ(\tau^{-1} + 1)$ and 
$\sigma\tau$ commute.  Hence the automorphism $\vartheta$ of $\cC$ commutes with $\sigma\tau$. Then we get an induced automorphism $\wt{\vartheta}$ of $\La$.
It follows that
\[
	\mu = \eta\circ \wt{\tau}^{-1} \circ\wt{\vartheta}
\]
and we have
\[
	\gamma = \eta^{-1} \circ \mu = \eta^{-1} \circ \eta\circ \wt{\tau}^{-1}\circ\wt{\vartheta} = \wt{\sigma} \circ \wt{\vartheta},
\]
and the automorphism $\wt{\vartheta}$ of $\La$ is inner. 
\end{proof}

\subsection{$\Ext^1$-symmetry}

In \cite[Section~5]{GLS-3} mutation of rigid modules in preprojective algebras of simply laced Dynkin type is studied. The results presented there heavily rely on the fact that the category of finitely generated modules over a preprojective algebra is stably 2-Calabi-Yau, which affords a symmetry of $\Ext^1$-spaces that we do not in general find in our $\modu \Lambda$. We can however still exploit some of the methods from \cite[Section~5]{GLS-3} in our situation, relying on the following symmetry, which holds in particular when $A$ is $\La$, one of our mesh algebras.

\begin{proposition}\label{P:Ext-symmetry}
	Let $A$ be as in Lemma \ref{L:Ext twist}, and let $X$ and $Y$ be modules in $\modu A$. Assume that $Y$ is $\gamma$-equivariant. Then we have
	\[
		D\Ext^1_A(Y,X) \cong \Ext^1_A(X,Y).
	\]
\end{proposition}

\begin{proof}
	This follows directly from Lemma \ref{L:Ext twist}.
\end{proof}

\section{Mutation of rigid modules in mesh algebras}\label{S:Mutation}

Throughout this section, let $\La$ be a finite dimensional self-injective algebra with automorphism $\mu$ such that 
$\Omega^3_{\La^e}(\La) \cong {}_{\mu} \La _{\mathrm{id}}$ and Nakayama automorphism $\eta$. 
As before, we set $\gamma = \eta^{-1} \circ \mu$. Assume further that any basic maximal rigid module in $\modu \La$ is $\gamma$-equivariant. In particular, we will show in Theorem \ref{T:short list} that we can choose $\La$ to be any mesh algebra from the following list:
\begin{enumerate}
			\item{$P(\bB_k)$ for $k \geq 2$},
			\item{$P(\bC_n)$ for $n \geq 3$},
			\item{$P(\bF_4)$}.
		\end{enumerate}
Then we obtain the following result, as a corollary of Proposition \ref{P:Ext-symmetry}.

\begin{corollary}\label{C:Ext-symmetry}
	Let $T$ be a basic maximal rigid module in $\modu \La$. Then for every module $X$ in $\modu \La$ we have
	\[
		D\Ext^1_\La(T,X) \cong \Ext^1_\La(X,T).
	\]
\end{corollary}

\subsection{Equivariance of maximal rigid modules for mesh algebras}
Asking that ${}_\gamma T \cong T$ for any basic maximal rigid $\La$-module $T$ is not some elusive condition on the algebra $\La$: In fact, this holds for the majority of mesh algebras, as we show below.

\begin{lemma}\label{L:maximal rigids are invariant}
Assume
that $\gamma^2 = \mathrm{id}$.
Then every basic maximal rigid module $T$ in $\modu \La$ is $\gamma$-equivariant.
\end{lemma}

\begin{proof}
	Assume as a contradiction that $T$ is not $\gamma$-equivariant, and write $T = \bigoplus_{j \in I} T_j$ with $T_j$ indecomposable. Then there is some $i \in I$ such that ${}_\gamma T_i \ncong T_j$ for all $j \in I$. By Lemma \ref{L:Ext twist} we have
	\[
		\Ext^1_\La(T, {}_\gamma T_i) \cong D\Ext^1_\La(T_i,T) = 0.
	\]
Furthermore, since $\gamma^2 = \mathrm{id}$ we have
	\[
		\Ext^1_\La({}_\gamma T_i, T) \cong \Ext^1_\La(T_i, {}_\gamma T) \cong D\Ext^1_\La(T,T_i) = 0.
	\]
Clearly, we have that ${}_\gamma T_i$ is rigid, and it follows that $T \oplus {}_\gamma T_i$ is rigid; a contradiction to the assumption that $T$ is maximal rigid.
\end{proof}

\begin{thm}\label{T:short list}
	Let $A$ be a mesh algebra from the following list:
\begin{enumerate}
			\item{$P(\bB_k)$ for $k \geq 2$},
			\item{$P(\bC_n)$ for $n \geq 3$},
			\item{$P(\bF_4)$}.
		\end{enumerate}
Then every basic maximal rigid module $T$ in $\modu A$ is $\gamma$-equivariant.
\end{thm}

\begin{proof}
	By Lemma \ref{L:order 2} the automorphism $\gamma$ has order $2$. The claim follows from Lemma \ref{L:maximal rigids are invariant}.
\end{proof}

We can thus choose our algebra $\La$ with the desired conditions listed at the start of Section \ref{S:Mutation}, to be one of the mesh algebras from Theorem \ref{T:short list}. Note that the only mesh algebra of non-simply laced Dynkin type that is not included in the list, is $P(\bG_2)$. There, the automorphism $\gamma$ has order $3$, and we cannot apply Lemma \ref{L:maximal rigids are invariant}.

\subsection{Minimal $\gamma$-equivariant modules}

Mutation of maximal rigid modules of preprojective algebras as studied in \cite{GLS-3} replaces an indecomposable summand of the maximal rigid module:  \cite[Proposition~6.7]{GLS-3} states that if $T \oplus X$ is a basic maximal rigid module of a preprojective algebra $P(\Delta)$ of Dynkin type, with $X$ indecomposable, then there exists a unique indecomposable module $Y \ncong X$ of $P(\Delta)$ such that $T \oplus Y$ is again a basic maximal rigid module. The following observation makes clear why we cannot expect to be able to mutate at a single indecomposable summand in our algebra $\La$, unless this summand is $\gamma$-equivariant.

\begin{lemma}\label{L:justification}
	Assume $T \oplus X$ is a basic maximal rigid module in $\modu \La$, such that $X$ is indecomposable and {\em not} $\gamma$-equivariant. Then there cannot exist a module $Y \ncong X$ such that $T \oplus Y$ is also basic maximal rigid.
\end{lemma}

\begin{proof}
Assume as a contradiction there exists a module $Y \ncong X$, such that $T \oplus Y$ is basic maximal rigid. Then we have non-trivial extensions between $X$ and $Y$
by maximal rigidity of $T \oplus X$, and thus also $\Ext_\Lambda^1({}_\gamma Y, {}_\gamma X) \neq 0$ or $\Ext^1_\La({}_\gamma X,{}_\gamma Y) \neq 0$. By our assumptions on $\La$, any basic maximal rigid module is $\gamma$-equivariant. So, since $T \oplus X$ is basic, and $X$ is indecomposable and not $\gamma$-equivariant, we must have that ${}_\gamma X$ is a summand of $T$. However, by $\gamma$-equivariance of $T \oplus Y$, the module ${}_\gamma Y$ must also be a summand of $T \oplus Y$, contradicting rigidity of $T \oplus Y$. 
\end{proof}

Instead, in our algebra $\Lambda$ we want to mutate with respect to {\em minimal $\gamma$-equivariant summands}. We introduce this concept more generally. Let $A$ be a finite dimensional self-injective algebra.

\begin{notation}
	Let $X$ be a module in $\modu A$. Then we write
	\[
		|X| = \; \text{number of indecomposable summands of $X$ up to isomorphism}.
	\]
\end{notation}

\begin{definition}\label{D:minimal equivariant}
	Let $X$ be a module in $\modu A$ and let $\varphi$ be an automorphism of $A$. We say that $X$ is {\em minimal $\varphi$-equivariant}, if it is basic and the indecomposable summands of $X$ form a $\varphi$-orbit, i.e.\ setting $n = |X|-1$ we have 
		\[
			X \cong \bigoplus_{i=0}^{n} {}_{\varphi^i} \tilde{X},
		\]
	where $\tilde{X}$ is an indecomposable $A$-module with ${}_{\varphi^{n+1}}\tilde{X} \cong \tilde{X}$.
\end{definition}

\begin{remark}\normalfont
	In the following, we will consider minimal $\gamma$-equivariant modules of $\La$. Note that if $\La$ is a mesh algebra as in Theorem \ref{T:short list}, then since $\gamma^2 = \mathrm{id}$ by Lemma \ref{L:order 2}, any minimal $\gamma$-equivariant module in $\modu \Lambda$ has one or two indecomposable summands.
\end{remark}

\subsection{Mutation of maximal rigid modules}\label{S:mutation maximal rigids}

Our strategy to construct mutations is based upon modifying \cite[Section~5]{GLS-3} by using our results on equivariance. For a brief reminder on approximations, we refer the reader to the preliminaries in \cite[Section~3.1]{GLS-3}.

\begin{lemma} \label{L:mutation is rigid}
	Let $T$ and $X$ be rigid $\Lambda$-modules with ${}_\gamma T \cong T$. 
	If
	\[
		\xymatrix{0 \ar[r] & X \ar[r]^-f & T' \ar[r]^-g & Y \ar[r] & 0}
	\]
	is a short exact sequence with $f$ a left $\add(T)$-approximation, then $T \oplus Y$ is rigid.
\end{lemma}

\begin{proof}
	This follows analogously to the proof of \cite[Lemma~5.1]{GLS-3} and exploiting the $\Ext^1$-symmetry from Proposition \ref{P:Ext-symmetry}.
\end{proof}

Recall that throughout this section, we have assumed that any basic maximal rigid module is $\gamma$-equivariant.
	
	\begin{corollary}\label{C:ses}
	Let $T$ and $X$ be rigid $\Lambda$-modules. If $T$ is basic maximal rigid then there exists a short exact sequence
	\[
		0 \to X \to T' \to T'' \to 0
	\]
	with $T'$ and $T''$ in $\add(T)$.
\end{corollary}

\begin{proof}
This follows analogously to the proof of \cite[Corollary~5.2]{GLS-3}, where we  use Lemma \ref{L:mutation is rigid}  and Corollary \ref{C:Ext-symmetry}.
\end{proof}

\begin{corollary}\label{C:projdim}
	Let $T$ and $X$ be rigid $\Lambda$-modules and set $E = \mathrm{End}_\Lambda(T)$. If $T$ is basic maximal rigid, then
	\[
		\mathrm{proj.dim_E}(\Hom_\Lambda(X,T)) \leq 1.
	\]
\end{corollary}

\begin{proof}
This follows analogously to the proof of \cite[Corollary~5.3]{GLS-3}: Applying $\Hom_\Lambda(-,T)$ to the short exact sequence from Corollary \ref{C:ses} yields the projective resolution
\[
	0 \to \Hom_\Lambda(T'',T) \to \Hom_\Lambda(T',T) \to \Hom_\Lambda(X,T) \to 0.
\]
\end{proof}

\begin{thm}\label{T:tilting module}
	Let $T_1$ and $T_2$ be basic maximal rigid $\Lambda$-modules. For $i = 1,2$ set $E_i = \mathrm{End}_\Lambda(T_i)$. Then $T = \Hom_\Lambda(T_2,T_1)$ is tilting over $E_1$ and $\mathrm{End}_{E_1}(T) \cong E_2^{op}$. In particular, the endomorphism algebras $\End_\La(T_1)$ and $\End_\La(T_2)$ are derived equivalent.
\end{thm}

\begin{proof}
Use Corollaries \ref{C:ses} and \ref{C:projdim}, and follow the proof of \cite[Theorem~5.3.2]{I} (cf.\ also the comments following Theorem 5.4 in \cite{GLS-3}).
\end{proof}

\begin{proposition}\label{P:exchange pair}
	Let $T \oplus X$ be a basic $\gamma$-equivariant rigid module in $\modu \Lambda$ such that
	\begin{itemize}
		\item{$X$ is minimal $\gamma$-equivariant;	}
		\item{$\Lambda \in \add(T)$.}
	\end{itemize}
	Then for $0 \leq i \leq n = |X|-1$ there exists a short exact sequence
	\[
		\xymatrix{S_i & = &0 \ar[r] & {}_{\gamma^i} \tilde{X} \ar[r]^-{{}_{\gamma^i} \tilde{f}} & {}_{\gamma^i} \tilde{T}' \ar[r]^-{{}_{\gamma^i} \tilde{g}}& {}_{\gamma^i} \tilde{Y} \ar[r] & 0}
	\]
	such that taking the direct sum yields an exact sequence
	\begin{eqnarray}\label{E:direct sum}
		\xymatrix{\bigoplus_{i=0}^{n}S_i & = & 0 \ar[r] & X \ar[r]^-f & T' \ar[r]^-g & Y \ar[r] & 0}
	\end{eqnarray}
	with the following properties:
	\begin{itemize}
		\item[(a)]{$f = \bigoplus_{i=0}^n {}_{\gamma^i} \tilde{f}$ is a minimal left $\add(T)$-approximation and $g = \bigoplus_{i=0}^n {}_{\gamma^i} \tilde{g}$ is a minimal right $\add(T)$-approximation.}
		\item[(b)]{$Y = \bigoplus_{i=0}^n {}_{\gamma^i} \tilde{Y}$ is minimal $\gamma$-equivariant, and $X \ncong Y$.}
		\item[(c)]{$T \oplus Y$ is basic $\gamma$-equivariant rigid.}
	\end{itemize}
\end{proposition}

\begin{proof}
	It follows analogously to the proof of \cite[Proposition 5.6]{GLS-3} that we have a short exact sequence
	\[
		\xymatrix{0 \ar[r] & \tilde{X} \ar[r]^-{\tilde{f}} & \tilde{T}' \ar[r]^-{\tilde{g}}& \tilde{Y} \ar[r] & 0}
	\]
	such that
	\begin{itemize}
		\item{$\tilde{f}$ is a minimal left $\add(T)$-approximation and $\tilde{g}$ is a minimal right $\add(T)$-approximation.}
		\item{$\tilde{Y}$ is indecomposable and $\tilde{X} \ncong \tilde{Y}$.}
		\item{$T \oplus \tilde{Y}$ is basic rigid.}
	\end{itemize}
	Consider now the basic module $X \cong \bigoplus_{i=0}^n {}_{\gamma^i}\tilde{X}$.	
	\begin{itemize}
	
	\item[(a)]
	Since $\tilde{f}$ is a minimal left $\add(T)$-approximation of $\tilde{X}$, we get that ${}_{\gamma^i} \tilde{f}$ is a minimal left $\add(T)$-approximation of ${}_{\gamma^i} \tilde{X}$ for $0 \leq i \leq n$, since by assumption $\add(T)$ is $\gamma$-equivariant. Since minimal approximations commute with direct sums, we obtain that $f = \bigoplus_{i=0}^n {}_{\gamma^i}\tilde{f}$ is a minimal left $\add(T)$-approximation. Dually, the map $g = \bigoplus_{i=0}^n {}_{\gamma^i}\tilde{g}$ is a minimal right $\add(T)$-approximation.
	
	\item[(b)]
	Since $X \notin \add(T)$, sequence (\ref{E:direct sum}) does not split, and since $X$ is rigid, it follows that $X \ncong Y$. Furthermore, we have ${}_{\gamma^i}\tilde{Y} \ncong {}_{\gamma^j} \tilde{Y}$ for any $0 \leq i<j \leq n$, or equivalently $\tilde{Y} \ncong {}_{\gamma^k} \tilde{Y}$ for any $1 \leq k \leq n$. Else, by uniqueness (up to isomorphism) of the minimal right $\add(T)$-approximation $\tilde{g}$ we would have $\tilde{T}' \cong {}_{\gamma^k} \tilde{T}'$, and thus isomorphic kernels $\tilde{X} \cong {}_{\gamma^k} \tilde{X}$; a contradiction to the assumption. By an analogous argument, we must have ${}_{\gamma^n}\tilde{Y} \cong \tilde{Y}$. Therefore, $Y = \bigoplus_{i=0}^n {}_{\gamma^i} \tilde{Y}$ is minimal $\gamma$-equivariant.

	\item[(c)]
	By Lemma \ref{L:mutation is rigid}, the module $T \oplus Y = T \oplus \bigoplus_{i=0}^n {}_{\gamma^i}\tilde{Y}$ is rigid. Furthermore, it is basic: We know that $T \oplus \tilde{Y}$ is basic. Analogously one can show that $T \oplus {}_{\gamma^i} \tilde{Y}$ is basic for any $0 \leq i \leq n$, and we have ${}_{\gamma^i}\tilde{Y} \ncong {}_{\gamma^j} \tilde{Y}$ for any $0 \leq i < j \leq n$. Since both $T$ and $Y$ are $\gamma$-equivariant, so is $T \oplus Y$.
	\end{itemize}	
\end{proof}

\subsection{Exchange pairs and exchange sequences}

The following definition is analogous to the definitions of exchange pair, exchange sequence, etc. in \cite[Section~5]{GLS-3}. Note that we keep track of the order of the two arguments $X$ and $Y$, whereas this is not needed in the classical definition.

\begin{definition}
	In the situation of Proposition \ref{P:exchange pair}, we call $(X = \bigoplus_{i=0}^n {}_{\gamma^i} \tilde{X},Y = \bigoplus_{i=0}^n {}_{\gamma^i} \tilde{Y})$ a {\em pointed exchange pair with base $(\tilde{X},\tilde{Y})$ } associated to $T$, and we call the sequence
	\[
		\xymatrix{0 \ar[r] & X \ar[r]^-f & T' \ar[r]^-g & Y \ar[r] & 0}
	\]
	the {\em exchange sequence} starting in $X$ and ending in $Y$. The module $T \oplus Y$ is called the {\em mutation of $T \oplus X$ in direction $X$} and we write
	\[
		\mu_X(T \oplus X) = T \oplus Y.
	\]
\end{definition}

\begin{proposition}\label{P:exchange backwards}
	Let $X = \bigoplus_{i=0}^{n} {}_{\gamma^i}\tilde{X}$ and $Y= \bigoplus_{i=0}^{n} {}_{\gamma^i}\tilde{Y}$ be basic rigid minimal $\gamma$-equivariant modules in $\modu \Lambda$  with $n = |X|-1 = |Y|-1$. 
 Assume that for $0 \leq i \leq n$ we have
\[
	\dim(\Ext^1_\Lambda(  \tilde{Y},  {}_{\gamma^i} \tilde{X})) = 	\begin{cases}
											1 & \text{if} \; i = 0\\
											0 & \text{if} \; i \neq 0
										\end{cases}
\] 
and let
	\begin{eqnarray}\label{E:indecomposable}
		\xymatrix{0 \ar[r] & \tilde{X} \ar[r]^-{\tilde{f}} & \tilde{M} \ar[r]^-{\tilde{g}} & \tilde{Y} \ar[r] & 0}
	\end{eqnarray}
	be a non-split exact sequence. Set
	\begin{eqnarray}\label{E:exchange sequence}
		\xymatrix{0 \ar[r] & X \ar[r]^-f & M \ar[r]^-g & Y \ar[r] & 0 \; \cong \; \bigoplus_{i=0}^n  \Big( 0 \ar[r] & {}_{\gamma^i} \tilde{X} \ar[r]^-{ {}_{\gamma^i}\tilde{f}} &  {}_{\gamma^i} \tilde{M} \ar[r]^-{ {}_{\gamma^i} \tilde{g}} &  {}_{\gamma^i} \tilde{Y} \ar[r] & 0 \Big). }
	\end{eqnarray}
	Then $M \oplus X$ and $M \oplus Y$ are rigid, and $X, Y \notin \add(M)$. If additionally there exists a module $T$ in $\modu \Lambda$ such that $T \oplus X$ and $T \oplus Y$ are basic maximal rigid, then $f$ is a minimal left $\add(T)$-approximation and $g$ is a minimal right $\add(T)$-approximation.
\end{proposition}

\begin{proof}
	{\bf $X,Y \notin \add(M)$:} For this part, we use a similar argument to the proof of \cite[Lemma~5.10]{GLS-3}. Assume $X \in \add(M)$. Since $X$ is basic, we get $M \cong X \oplus M'$ for some $\Lambda$-module $M'$ and Sequence (\ref{E:exchange sequence}) reads as
	\[
		0 \to X \to X \oplus M' \to Y \to 0.
	\]
By Riedtmann \cite[Proposition~3.4]{Riedtmann} $M'$ degenerates to $Y$. Since $Y$ is rigid, this implies $M' = Y$. Thus the above sequence splits; a contradiction. Dually one shows that $Y \notin \add(M)$.

{\bf $\Ext_\Lambda^1(M, X \oplus Y)=0$:} For this part, we use a similar argument to the proof of \cite[Lemma~5.11]{GLS-3}. Apply $\Hom_\Lambda(-,\tilde{X})$ to sequence (\ref{E:exchange sequence}). This yields an exact sequence
\[
	\xymatrixcolsep{2.5pc}\xymatrix{0 \ar[r] & \Hom_\Lambda(Y,\tilde{X}) \ar[r] & \Hom_\Lambda(M,\tilde{X}) \ar[r]^-{\Hom_\Lambda(f,\tilde{X})} & \Hom_\Lambda(X,\tilde{X}) \\
				  \ar[r]^-\delta 				& \Ext^1_\Lambda(Y,\tilde{X}) \ar[r] 				& \Ext^1_\Lambda(M,\tilde{X}) \ar[r]  & \Ext^1_\Lambda(X,\tilde{X}) = 0. }
\]
Suppose that $\Hom_\Lambda(f,\tilde{X})$ is surjective. Then the projection $X \to \tilde{X}$ factors through $f \colon X \to M$, and $\tilde{X} \in \add(M)$. However, since $M$ is $\gamma$-equivariant, this would imply $X \in \add(M)$ contradicting the above observation. So $\delta \neq 0$. Since $\dim \Ext^1_\Lambda(Y,\tilde{X}) = 1$ this implies that $\delta$ is surjective. Thus, since $M$ is $\gamma$-equivariant, we obtain $\Ext^1_\Lambda(M, {}_{\gamma^i} \tilde{X}) \cong \Ext^1_\Lambda(M,\tilde{X}) = 0$ for all $0 \leq i \leq n$. Dually one proves that $ \Ext^1_\Lambda(M, {}_{\gamma^i} \tilde{Y}) =0$ for all $0 \leq i \leq n$.

{\bf $\Ext^1_\Lambda(M,M) = 0$:} Applying $\Hom_\Lambda(-,M)$ to sequence (\ref{E:exchange sequence}) yields an exact sequence
\[
	\xymatrix{0 \ar[r] & \Hom_\Lambda(Y,M)) \ar[r] & \Hom_\Lambda(M,M) \ar[r] & \Hom_\Lambda(X,M) \\
				  \ar[r]				& \Ext^1_\Lambda(Y,M) \ar[r] 				& \Ext^1_\Lambda(M,M) \ar[r]  & \Ext^1_\Lambda(X,M). }
\]
Since $X$ and $Y$ are rigid $\gamma$-equivariant, by Proposition \ref{P:Ext-symmetry} and since $\Ext^1(M,X\oplus Y) = 0$, we have $\Ext_\Lambda^1(Y,M) \cong \Ext_\Lambda^1(M,Y) = 0$, as well as $\Ext_\Lambda^1(X,M) \cong \Ext_\Lambda^1(M,X) = 0$. Therefore $\Ext^1_\Lambda(M,M) = 0$.

Now assume additionally that there exists a module $T$ in $\modu\Lambda$ such that $T \oplus X$ and $T \oplus Y$ are basic maximal rigid.

{\bf $\Ext^1_\Lambda(M, T) = 0$:} This follows analogously to the proof of \cite[Lemma~5.12]{GLS-3}, by applying $\Hom_\Lambda(-, T)$ to Sequence (\ref{E:exchange sequence}). 

{\bf $M \in \add(T)$:} This follows analogously to the proof of \cite[Lemma~5.13]{GLS-3}.

Since $\Ext^1_\Lambda(M, T) = 0$ and $M \in \add(T)$, it follows that $f \colon X \to M$ is a minimal left $\add(T)$-approximation. Dually, the morphism $g \colon M \to Y$ is a minimal right $\add(T)$-approximation. This concludes the proof.
\end{proof}

\begin{corollary}\label{C:involutive}
	Let $(X = \bigoplus_{i=0}^n \:_{\gamma^i} \tilde{X},Y= \bigoplus_{i=0}^n {}_{\gamma^i} \tilde{Y})$ be a pointed exchange pair with base $(\tilde{X},\tilde{Y})$ associated to a basic rigid module $T$, such that $T \oplus X$ and $T \oplus Y$ are maximal rigid. Assume further that for $0 \leq i \leq n$ 
	\[
		\dim(\Ext^1_\Lambda( \tilde{Y}, {}_{\gamma^i}\tilde{X})) = 	\begin{cases}
											1 & \text{if} \; i=0\\
											0 & \text{if} \; i \neq 0.
										\end{cases}
	\]
	Then we have
	\[
		\mu_Y(\mu_X(T \oplus X)) = T \oplus X.
	\]
\end{corollary}

\begin{proof} 
Let
	\[
		0 \to X \to T' \to Y \to 0 \cong \bigoplus_{i=0}^n \Big( 0 \to  {}_{\gamma^i} \tilde{X} \to {}_{\gamma^i} \tilde{T}' \to  {}_{\gamma^i} \tilde{Y} \to 0 \Big)
	\]
	be the short exact sequence from Proposition \ref{P:exchange pair}, so we have $\mu_X(T \oplus X) = T \oplus Y$. Note that by Lemma \ref{L:Ext twist} we have
\[
	D\Ext^1(\tilde{X}, {}_{\gamma^i}\tilde{Y}) \cong \Ext^1({}_{\gamma^i} \tilde{Y}, {}_{\gamma}\tilde{X}) \cong \Ext^1(\tilde{Y}, {}_{\gamma^{1-i}}\tilde{X}).
\]
Therefore, by assumption, for $0 \leq i \leq n$ we have
\[
	\dim \Ext^1(\tilde{X}, {}_{\gamma^i}\tilde{Y}) =	\begin{cases} 
												1 & \text{if} \; i = 1 \\
												0 & \text{if} \; i \neq 1.	
											\end{cases}
\]
Further, since $T \oplus X$ and $T \oplus Y$ are basic maximal rigid, Proposition \ref{P:exchange backwards} yields a non-split short exact sequence
	\[
		\xymatrix{0 \ar[r] & Y\ar[r]^-h & M \ar[r] & X \ar[r] & 0 \cong \bigoplus_{i=0}^n \Big( 0 \ar[r] & {}_{\gamma^{i+1}} \tilde{Y} \ar[r] & {}_{\gamma^i} \tilde{M} \ar[r] & {}_{\gamma^i} \tilde{X} \ar[r] & 0 \Big)}
	\]	
	where $h$ is a minimal left $\add(T)$-approximation. Thus $\mu_Y(T \oplus Y) = T \oplus X$. 
\end{proof}

\begin{remark}\label{R:twist} \normalfont
It follows from the Proof of \ref{C:involutive} that if $(X,Y)$ is a pointed exchange pair with base $(\tilde{X}, \tilde{Y})$ associated to $T$ such that $T \oplus X$ and $T \oplus Y$ are basic maximal rigid, then $(Y,X)$ is a pointed exchange pair with base $({}_\gamma \tilde{Y}, \tilde{X})$ associated to $T$. Namely, when we exploit $\Ext^1$-symmetry for $\gamma$-equivariant modules, we have to be mindful of what happens to their summands. Thus, for a pointed exchange pair $(X,Y)$ with base $(\tilde{X},\tilde{Y})$ as in Corollary \ref{C:involutive} where we have that
\[
	\mu_Y(\mu_X(T \oplus X)) = T \oplus X,
\]
the associated exchange sequences decompose in a different manner: We have
\begin{align}\label{S:forwards}
	0 \to X \to T' \to Y \to 0 \cong \; \bigoplus_{i=0}^n \Big(0 \to {}_{\gamma^i} \tilde{X} \to {}_{\gamma^i} \tilde{T}' \to  {}_{\gamma^i} \tilde{Y} \to 0 \Big),
\end{align}
for the exchange sequence starting in $X$ and ending in $Y$ yet the backwards mutation sees a twist by $\gamma$; we have
\begin{align}\label{S:backwards}
		0 \to Y\ \to M \to X \to 0 \cong \; \bigoplus_{i=0}^n \Big(0 \to {}_{\gamma^{i+1}} \tilde{Y} \to {}_{\gamma^i} \tilde{M} \to  {}_{\gamma^i} \tilde{X} \to 0 \Big)
	\end{align}	
for the exchange sequence starting in $Y$ and ending in $X$.
\end{remark}

\begin{remark}\label{R:generalise}\normalfont
More generally, throughout all previous results in this section we could replace $\La$ by any finite dimensional self-injective algebra $A$ with automorphism $\mu$ such that $\Omega_{A^e}^3(A) \cong {}_\mu A_{\mathrm{id}}$. While we do not know if every basic maximal rigid module is $\gamma$-equivariant, all of our previous results starting from Section \ref{S:mutation maximal rigids}, still apply if, whenever the term ``maximal rigid module'' appears, we replace it by ``$\gamma$-equivariant maximal rigid module'', and rely on Proposition \ref{P:Ext-symmetry} rather than Corollary \ref{C:Ext-symmetry}. 

In particular, we could pick $\La$ to be the remaining mesh algebra of Dynkin type, namely $P(\bG_2)$. It does have a $\gamma$-equivariant maximal rigid module, namely the cluster tilting module constructed in Theorem \ref{thm:La_2_ct}.
\end{remark}

\begin{remark}\label{R:IY mutation} \normalfont
	Note that in the stable module category, our mutations induce mutations in the sense of Iyama and Yoshino \cite[Section 5]{IY}. Indeed, in $\underline{\modu} \La$ we have Serre functor $\mathbb{S} = \Omega \mathcal{N}$. The functor denoted by $\mathbb{S}_n$ in \cite{IY}, which is of interest in our case for $n=2$, takes on the form
	\[
		\mathbb{S}_2 = \mathcal{N}\Omega^3.
	\]
Thus for any $M \in \underline{\modu}\La$ we have 
	\[
		\mathbb{S}_2 M \cong {}_\gamma M.
	\]
Note that, in addition to being defined already on the abelian category of modules, our mutation theory for our specific algebras is more explicit, in that it describes precisely which summands can be exchanged, and how mutation induces a twist on the indecomposable summands.
\end{remark}

The following shows that one can mutate a cluster tilting module at any minimal $\gamma$-equivariant non-projective summand.

\begin{proposition}\label{P:unique}
	Let $T$ be a basic cluster tilting module in $\modu \La$, and let $X$ be a minimal $\gamma$-equivariant direct summand of $T$. If $X$ is non-projective, then up to isomorphism there exists exactly one minimal $\gamma$-equivariant $Y$ in $\modu \La$ such that $X \ncong Y$ and $Y \oplus T/X$ is cluster tilting.
\end{proposition}

\begin{proof}
This follows directly from Remark \ref{R:IY mutation} together with \cite[Theorem~5.3]{IY}.
\end{proof}

\section{Endomorphism algebras of maximal rigid modules}\label{S:Global dimension}

Throughout this section, let $\La$ be a finite dimensional self-injective algebra with automorphism $\mu$ such that $\Omega^3_{A^e}(A) \cong {}_\mu A_{\mathrm{id}}$ and with Nakayama automorphism $\eta$.
As before, we set $\gamma = \eta^{-1} \circ \mu$. In line with Remark \ref{R:generalise}, we do not need to know whether every basic maximal rigid module is $\gamma$-equivariant.

In the following, we study endomorphism algebras and their quivers. The strategy of this section, and the beginning of the next, is analogous to Section \cite[Section~6]{GLS-3}. If $M$ is a module in $\modu \Lambda$, then the quiver of $\End_\Lambda(M)$ has vertices labelled by the indecomposable modules in $\add(M)$. By abuse of notation, we will denote the vertex associated to an indecomposable $\tilde{M}$ in $\add(M)$ by $\tilde{M}$ as well.

\begin{lemma}\label{L:no loops}
	Let $(X = \bigoplus_{i=0}^n \:_{\gamma^i} \tilde{X},Y= \bigoplus_{i=0}^n {}_{\gamma^i} \tilde{Y})$ be a pointed exchange pair with base $(\tilde{X},\tilde{Y})$ associated to a basic rigid module $T$ in $\modu \La$. The following are equivalent:
	\begin{enumerate}
		\item{The quiver of $\End_\La(T \oplus X)$ has no arrows from $\tilde{X}$ to ${}_{\gamma^i}\tilde{X}$ for all $i \in \bZ$.}
		\item{Every radical map $X \to X$ factors through $\add(T)$.}
		\item{For $0 \leq i \leq n$ we have
		\[
			\dim \Ext^1_\La(\tilde{Y}, {}_{\gamma^i}\tilde{X}) = 	\begin{cases}
														1 & \text{if} \; i = 0\\
														0 & \text{else}.
													\end{cases}
		\]}
	\end{enumerate}
\end{lemma}

\begin{proof} 
	We first show that (1) is equivalent to (2), and then that (2) is equivalent to (3).
	
	{\bf (1) $\Rightarrow$ (2):} Assume that $f \colon X \to X$ is a radical map which does not factor through $\add(T)$. Then, in the quiver of $\End_\La(T \oplus X)$, this yields an arrow from ${}_{\gamma^i}\tilde{X}$ to ${}_{\gamma^j}\tilde{X}$ for some $i,j \in \bZ$, and thus an arrow from $\tilde{X}$ to ${}_{\gamma^{j-i}}\tilde{X}$.
	
	{\bf (2) $\Rightarrow$ (1):} If, in the quiver of $\End_\La(T \oplus X)$, we have an arrow from $\tilde{X}$ to ${}_{\gamma^i}\tilde{X}$, it is in the radical of $\End_\La(T \oplus X)$. Furthermore, it is not in $\rad^2(\End_\La(T \oplus X))$, so it cannot factor through $\add(T)$.
	
	{\bf (2) $\Rightarrow$ (3):} Since $(X,Y)$ is a pointed exchange pair with base $(\tilde{X},\tilde{Y})$, there is a short exact sequence
	\begin{equation} \label{E:indecomposable exchange sequence}
		\xymatrix{0 \ar[r] & \tilde{X} \ar[r]^-{\tilde{f}} & \tilde{T}' \ar[r]^-{\tilde{g}} & \tilde{Y} \ar[r] & 0}	
	\end{equation}
	with $\tilde{T}' \in \add(T)$. Applying $\Hom_\La(-,X)$ yields the exact sequence
	\begin{equation} \label{E:exact sequence from indecomposable exchange sequence}
		\xymatrix{0 \ar[r] & \Hom_\La(\tilde{Y},X) \ar[r] & \Hom_\La(\tilde{T}',X) \ar[rr]^-{\Hom_\La(\tilde{f},X)} && \Hom_\La(\tilde{X},X) \ar[r] & \Ext_\La^1(\Tilde{Y},X) \ar[r] & 0}	
	\end{equation}
	By assumption any radical map $h \colon \tilde{X} \to X$ must factor through $\add(T)$, so we must have $h = \mu \circ \theta$ for some $\theta \colon \tilde{X} \to T''$ and $\mu \colon T'' \to X$ with $T'' \in \add(T)$. Now, since $\tilde{f}$ is a left $\add(T)$-approximation, we have $\theta = \psi \circ \tilde{f}$ for some $\psi \colon \tilde{T}' \to T''$, and we obtain a factorization 
	\[
		h = \mu \circ \psi \circ \tilde{f}.
	\]
	This means that $\mathrm{coker}(\Hom_\La(\tilde{f},X))$ is $1$-dimensional; it is spanned by the coset of the inclusion of $\tilde{X}$ into $X$. Therefore, we have
	\[
		\dim \Ext^1_\La(\tilde{Y},X) = 1.
	\]
	Now, by existence of the short exact sequence (\ref{E:indecomposable exchange sequence}) we know that $\Ext_\La^1(\tilde{Y},\tilde{X}) \neq 0$ and the claim follows.
	
	{\bf (3) $\Rightarrow$ (2):} Consider the sequence (\ref{E:exact sequence from indecomposable exchange sequence}). By assumption, $\dim \Ext^1(\tilde{Y},X)=1$. The inclusion $\iota \colon \tilde{X} \to X$ does not factor through $\tilde{f}$, since the short exact sequence (\ref{E:indecomposable exchange sequence}) does not split. So $\mathrm{coker}(\Hom_\La(\tilde{f},X))$ is spanned by the coset of $\iota$. The claim follows.
\end{proof}

Let now $T$ be a basic $\gamma$-equivariant maximal rigid module in $\modu \La$. Note that this is a stronger assumption than we had for $T$ previously, where it was just assumed to be basic rigid. If $\La$ is a mesh algebra of Dynkin type other than $P(\bG_2)$, or any $\La$ where $\gamma$ has order at most $2$, we can choose $T$ to be any basic maximal rigid module by Lemma \ref{L:maximal rigids are invariant}.
Consider the quiver of $E = \End_\La(T)$. We denote by $S_{\tilde{X}}$ the simple $E$-module associated to an indecomposable $\tilde{X} \in \add(T)$. 

\begin{proposition}\label{P:gldim}
	If the quiver of $E$ has no arrows from $\tilde{X}$ to ${}_{\gamma^i}\tilde{X}$ for all indecomposable modules $\tilde{X} \in \add(T)$ and all $i \in \bZ$, then
	\[
		\mathrm{gl.dim}(E)=3.
	\]
\end{proposition}

\begin{proof}
	Consider an indecomposable module $\tilde{X}$ in $\add(T)$. Assume first that $\tilde{X}$ is non-projective. Let $X$ be the minimal $\gamma$-equivariant summand of $T$ having $\tilde{X}$ as a summand:
	\[
		X = \bigoplus_{i=0}^{|X|-1}{}_{\gamma^i}\tilde{X}.
	\]
	Consider the exchange pair $(X,Y)$ with base $(\tilde{X},\tilde{Y})$ associated to $T/X$. Since $\tilde{X}$ is non-projective, the module $X$ has no projective summands, and since $T$ is maximal rigid, we have $\La \in \add(T/X)$.
	
	By the discussions in Remark \ref{R:twist}, based on Propositions \ref{P:exchange pair} and \ref{P:exchange backwards}, we have short exact sequences
	\[
		\xymatrix{0 \ar[r] & \tilde{X} \ar[r]^-{\tilde{f}} & \tilde{T}' \ar[r] & \tilde{Y} \ar[r] & 0}
	\]
	and
	\[
		\xymatrix{0 \ar[r] & \tilde{Y} \ar[r]& \tilde{T}'' \ar[r] & {}_{\gamma^{-1}}\tilde{X} \ar[r] & 0}
	\]
with $\tilde{T}'$ and $\tilde{T}''$ in $\add(T/X)$. Applying $\Hom_\Lambda(-,T)$ we obtain sequences
	\[
		0 \to \Hom_\La(\tilde{Y},T) \to \Hom_\La(\tilde{T}',T)\xrightarrow{\Hom_\La(\tilde{f},T)} \Hom_\La(\tilde{X},T) \to \Ext_\La^1(\tilde{Y},T) \to 0,
	\]
	where $\Ext_\La^1(\tilde{Y},T) \cong \Ext_\La^1(\tilde{Y},\tilde{X})$ by Lemma \ref{L:no loops}, and
	\[
		0 \to \Hom_\La({}_{\gamma^{-1}}\tilde{X},T) \to \Hom_\La(\tilde{T}'', T)\to \Hom_\La(\tilde{Y},T) \to \Ext_\La^1({}_{\gamma^{-1}} \tilde{X},T) = 0.
	\]
The cokernel of $\Hom_\La(\tilde{f},T)$, i.e. $\Ext^1_\La(\tilde{Y},\tilde{X})$, is one-dimensional by Lemma \ref{L:no loops}. Therefore it is isomorphic to $S_{\tilde{X}}$, as $\Hom_\La(\tilde{X},T)$ is the projective of $E$ associated to $\tilde{X}$. Combining the two sequences yields an exact sequence
\begin{align}\label{S:resolution of SX}
	0 \to \Hom_\La({}_{\gamma^{-1}}\tilde{X},T) \to \Hom_\La(\tilde{T}'',T) \to \Hom_\La(\tilde{T}',T) \to \Hom_\La(\tilde{X},T) \to S_{\tilde{X}} \to 0.
\end{align}
This is a projective resolution of $S_{\tilde{X}}$. Applying $\Hom_E(-,S_{({}_{\gamma^{-1}} \tilde{X})})$ and using elementary homological algebra yields
\[
\Ext^3_E(S_{\tilde{X}}, S_{({}_{\gamma^{-1}} \tilde{X})}) \cong \Hom_E(\Hom_{\La}(_{\gamma^{-1}}\wt{X}, T), S_{({}_{\gamma^{-1}} \tilde{X})}) \neq 0,
\]
and hence $\mathrm{pdim} S_{\tilde{X}} = 3$.

Next assume that $\wt{X}$ is projective, let $Z:= \wt{X}/S$ where $S$ is the (simple) socle of $\wt{X}$. 
We will show (below) that
\begin{enumerate}
\item There is a short exact sequence 
\[
	0\to Z \xrightarrow{f} T'\to \bar{Z}\to 0
\]
where %$T'\in {\rm add}(T/X)$ and 
$f$ is  a minimal left add$(T/X)$ approximation;
\item $\bar{Z}$ is in $\add(T)$. 
\end{enumerate}
Assume these for the moment, then we let $h = f\circ \pi$ where $\pi: \wt{X}\to Z$ is the canonical epimorphism. 
We apply the functor $(-, T) = {\rm Hom}_{\La}(-, T)$ to the exact sequence
\[
	\wt{X} \xrightarrow{h} T'\to \bar{Z}\to 0.
\]
This gives an exact sequence of $E$-modules,
\[
	0\to (\bar{Z}, T)\to (T', T) \xrightarrow{\Hom(h,T)} (\wt{X}, T)\to W\to 0
\]
By (1) and (2), this is a projective resolution of the $E$-module $W$. It is now enough to show
that $\dim W =1$,  so that  $W= S_{\wt{X}}$ and hence $\mathrm{pdim}_E(S_{\wt{X}})\leq 2$.

We show that any radical map in $(\wt{X}, T)$ factors through $h$. 
Take $\psi: \wt{X}\to T''$ where $T''$ is some indecomposable summand of $T$, and where $\psi$ is not an isomorphism. 
Note that $\wt{X}$ is also an injective module, so  if $\psi$ were injective then it would split and $\wt{X}\cong T''$, which contradicts
the assumption.  It follows that $\psi(S)=0$ (since $S$ is the simple socle of $\wt{X}$).
Therefore there is $\bar{\psi}: Z\to T''$ such that $\psi = \bar{\psi}\circ \pi$. Now we use that
$f$ is the left approximation, which implies that $\bar{\psi}$ factors through $f$, say $\bar{\psi} = \eta\circ f$.  Combining these shows that
$\psi$ factors through $h$, as required.

It remains to prove (1) and (2). As before, let $X$ be the minimal $\gamma$-equivariant summand of $T$ which has summand $\wt{X}$.
Now let $f: Z\to T'$ be the minimal add$(T/X)$ approximation, we show that this is injective: We must show that the injective hull of $Z$ is in add$(T/X)$, i.e.\ does not have a summand in $\add(X)$. 
Suppose this is false, then $Z$ has a simple submodule $L$ isomorphic to the socle of $_{\gamma^j}\wt{X}$ for some $j$ which is $L\cong _{\gamma^j}S$. 
Then $\wt{X}$ has a submodule of length two with composition factors $S$ and $_{\gamma^j}S$. That is, there is an arrow in the quiver of $\La$ between the corresponding vertices.
From the presentation of $\La$ (see for example \cite{ES}), this is not the case and we have a contradiction.
%{\blue{maybe we can cut this short}}
This proves (1), setting $\bar{Z}$ to be the quotient $X/{\rm Im} (f)$.

We claim that $Z$ is rigid.  Indeed we have
\[
	{\rm Ext}^1_{\La}(Z, Z)\cong \  {\rm Ext}^1_{\La}(\Omega(Z), \Omega(Z)) =  {\rm Ext}^1_{\La}(S, S) = 0.
\]
By Lemma \ref{L:mutation is rigid} and (1),  we have that  $\bar{Z}\oplus T/X$ is rigid.
Now, $X$ is projective and therefore $\bar{Z}\oplus T/X\oplus X = \bar{Z}\oplus T$ is also rigid. 
But $T$ is maximal rigid, hence $\bar{Z}$ is in $\add(T)$.
\end{proof}

\section{Endomorphism algebras of cluster tilting modules for mesh algebras} \label{S:rep dimension}

We now return our focus on mesh algebras. Throughout this section let $\La$ denote a mesh algebra of non-simply laced Dynkin type, and as before let $\gamma$ be the automorphism of $\La$ described in Section \ref{S:gamma}. Note that in all cases but $P(\bG_2)$, any basic maximal rigid module is automatically $\gamma$-equivariant, and we could remove this assumption from the statements of the results in these cases.

\begin{thm}\label{T:maximal rigid = 2-ct}
	Let $T$ be a basic $\gamma$-equivariant maximal rigid module in $\modu \La$, and set $E = \End_\La(T)$. Then the following hold:
	\begin{enumerate}
		\item{The quiver of $E$ has no arrows from $\tilde{X}$ to ${}_{\gamma^i} \tilde{X}$ for any indecomposable $\tilde{X} \in \add(T)$.}
		\item{We have $\mathrm{gl.dim}(E) = 3$.}
		\item{The module $T$ is cluster tilting.}
		\item{We have $\mathrm{dom.dim}(E) = 3$.}
	\end{enumerate}
\end{thm}

Before we provide the Proof of Theorem \ref{T:maximal rigid = 2-ct}, let us recall some facts from Sections \ref{S:Preliminaries} and \ref{S:Existence}. In the proof of Theorem \ref{T:maximal rigid = 2-ct} we will consider the module
	\[
		I_\La = \tilde{F}_*(\bigoplus_{x \in A_\Delta}I(x))
	\]
	in $\modu \La$ from the proof of Theorem \ref{thm:La_2_ct}, where $\tilde{F}_* \colon \modu \cC \to \modu \La$ as usual denotes the push-down of the covering functor. Recall that $I(x)$ are the indecomposable injective $\Gamma_\Delta$-modules, viewed as $\cC$-modules, and
	 that $I_\La$ is the push-down of the pull-up of the start module $I_\Delta$ in $P(\Delta)$. It is a cluster tilting module, and we think of it as our ``start module'' in $\modu \La$. Note that, in particular, our start module $I_\La$ in $\modu \La$ is basic $\gamma$-equivariant maximal rigid.

\begin{proof}
	By Theorem \ref{T:tilting module}, the endomorphism algebra $\End_\La(T)$ is derived equivalent to $\End_\La(I_\La)$, the endomorphism algebra of our start module $I_\La$ in $\modu \La$ (for $P(\bG_2)$ keep in mind Remark \ref{R:generalise}). Now, since $I_\La$ is cluster tilting we have by \cite[Theorem~0.2]{I} that
	\[
		\mathrm{gl.dim} \End_\La(I_\La) \leq 3 < \infty.
	\]
This implies that $\End_\La(T)$ has finite global dimension. Therefore, by Igusa's work on the automorphism conjecture \cite[Theorem~3.2]{Igusa}, the quiver of $\End_\La(T)$ has no arrows from $\tilde{X}$ to ${}_{\gamma^i} \tilde{X}$ for every indecomposable module $\tilde{X} \in \add(T)$ and all $i \in \bZ$. Thus, by Proposition \ref{P:gldim} we have $\mathrm{gl.dim}(\End_\La(T)) = 3$. So $T$ is cluster tilting by \cite[Theorem~5.1(3)]{I}. Again by \cite[Theorem~0.2]{I} we get that $\mathrm{dom.dim(\End_\La(T))} = 3$, since $\mathrm{gl.dim}\La = 3$ also implies $\mathrm{dom.dim}(E) \leq 3$.
	
\end{proof}

\begin{corollary}
	The algebra $\La$ has representation dimension $\leq 3$. In particular, if it is of infinite type, it has representation dimension $3$.
\end{corollary}

\begin{corollary}\label{C:maximal rigid = 2-ct}
	Assume $\La$ is not of type $P(\bG_2)$. Then a basic module $T$ in $\modu \La$ is cluster tilting if and only if it is maximal rigid.
\end{corollary}

\begin{remark}
	Alternatively, Corollary \ref{C:maximal rigid = 2-ct} also follows from Lemma \ref{L:maximal rigids are invariant} together with \cite[Corollary~2.15]{YZZ}.
\end{remark}

\begin{proposition}\label{P:number of orbits}
Let $T$ be a basic cluster tilting module of $\La$. 
Then the number of minimal $\gamma$-equivariant summands 
         of $T$ is the number
         of positive roots of the corresponding root system, cf.\ Table \ref{tab:positive roots}.
\end{proposition}

\begin{proof}
We show, case by case, that the claim is true for $T = I_\La$, our start module. It then follows by an observation by Yang, Zhang and Zhu \cite[Corollary~2.15]{YZZ} that the claim holds for every cluster tilting module.
Observe that $\gamma$ acts as $\sigma$ on the summands of our start module (cf.\ Lemma \ref{L:order 2}) so what we want to count is the number of $\sigma$-orbits on $A_\Delta$.
\begin{enumerate}
\item Assume $\La = P(\bB_k)$ for $k \geq 2$. Then we start with the quiver $\Delta$ of type $\bA_{2k-1}$ from Figure \ref{fig:orientation}, and form $A_\Delta$ which consists of $k$ 
copies of $\Delta$ suitably connected. The automorphism $\sigma$ fixes each copy of $\Delta$; on it fixes the central vertex and has $k-1$ orbits of length 2. In total, the automorphism $\sigma$ has $k \cdot (1+(k-1)) = k^2$ orbits.
\item Assume $\La = P(\bC_n)$ for $n \geq 3$. Start with the quiver $\Delta$ of type $\bD_{n+1}$ from Figure \ref{fig:orientation}. Then $A_\Delta$ consists of $n$ copies of $\Delta$ suitable connected. In each copy of $\Delta$, the automorphism $\sigma$ fixes $n-1$ vertices and has one orbit of length $2$. In total, it has $(1+(n-1))\cdot n = n^2$ orbits.
\item Assume $\La = P(\bF_4)$. Start with the quiver $\Delta$ of type $\bE_6$ from Figure \ref{fig:orientation}. Then $A_\Delta$ consists of six copies of $\Delta$ suitably connected. The automorphism $\sigma$ fixes each of these copies. On it, there are two fixed points and two orbits of length $2$, i.e.\ $4$ orbits. In total, there are $4 \cdot 6 = 24$ orbits.
\item Assume $\La = P(\bG_2)$. Start with the quiver $\Delta$ of type $\bD_4$ from Figure \ref{fig:orientation} (setting $n=3$). Then $A_\Delta$ consists of $3$ copies of $\Delta$ suitably connected. On each copy, the automorphism $\sigma$ has one orbit of length $3$ and one fixed point. In total, there are $3 \cdot (1+1) = 6$ orbits.
\end{enumerate}
\end{proof}

\begin{table}[]
\begin{tabular}{|l|l|l|l|l|l|l|l|}
\hline
Dynkin type &  $\bA_n$  & $\bB_k$ & $\bC_n$ & $\bD_n$  & $\bE_6 $& $\bF_4$  & $\bG_2$  \\ \hline
Number of positive roots& $\frac{1}{2}(n^2+n)$ & $k^2$  & $n^2$  & $n(n-1)$  & $36$ & $24$ & $6$ \\ \hline
\end{tabular}
\caption{Number of positive roots for the Dynkin types} \label{tab:positive roots}
\end{table}

As for preprojective algebras (cf.\ \cite[Conjecture~6.10]{GLS-3}), it is an open question whether any cluster tilting module of a mesh algebra is reachable from a fixed cluster tilting module. If this is the case, all our results on mutation in Section \ref{S:Mutation} go through unchanged for $P(\bG_2)$, since $\gamma$-equivariance is preserved under mutation. 

\section{Matrix mutation}\label{S:matrix mutation}

Throughout this section, let $\La$ be a finite dimensional self-injective algebra with automorphism $\mu$ such that 
$\Omega^3_{\La^e}(\La) \cong {}_{\mu} \La _{\mathrm{id}}$ and Nakayama automorphism $\eta$. As before, we write $\gamma = \eta^{-1} \circ \mu$. Let $T$ be a $\gamma$-equivariant basic cluster tilting module. We denote by $\min_\gamma(T)$ the set of minimal $\gamma$-equivariant summands of $T$.

As we have seen, we can mutate $T$ at any non-projective minimal $\gamma$-equivariant summand. Combinatorially, this suggests that rather than directly mutating the quiver of $\End_\La(T)$, respectively the adjacency matrix thereof, we want to mutate the matrix we obtain by taking the quotient by the automorphism group generated by $\gamma$. 

Suppose $T^*$ is obtained from $T$ by such a mutation, we will now show that the quivers of ${\rm End}(T)$ and of ${\rm End}(T^*)$ are related by a Fomin-Zelevinsky type mutation. This modifies  the work from \cite[Section~7]{GLS-3} to the skew-symmetrisable setting. Since, after taking quotients by the automorphism group generated by $\gamma$, we do no longer deal with skew-symmetric matrices, we must make the assumption on the cluster tilting modules to be {\em admissible} (formulated in \ref{D:admissible}). The start module of $\modu(\La)$ satisfies this condition, as do all $\gamma$-equivariant cluster tilting modules in $P(\bG_2)$ and $P(\bC_n)$ for $n \geq 3$ (cf.\ Section \ref{S:admissible}). From examples, it looks like cluster tilting modules related to the start module by mutation might be admissible for all mesh algebras $\La$. Note however, that we cannot expect from the general theory of matrix mutation that this always must hold, cf.\ for example Dupont's counterexample \cite[Remark~2.18]{Dupont-ArXiv}.

\begin{remark}
	Considering $\gamma$-equivariant modules over $\La$ amounts to considering $\sigma$-equivariant modules in a suitable preprojective algebra $A$ (where $\sigma$ is the automorphism on $A$ induced by the automorphism $\sigma$, as described in Section \ref{S:list}, on its Galois-cover $\cC$). We refer the reader to Demonet's work \cite{Demonet} for an approach to the combinatorics of skew-symmetrisable cluster algebras via $\sigma$-equivariant modules in $\modu \mathbb{K}A$. 
	
	In contrast, here we study mutations directly in the module category of the mesh algebras. In particular, this offers a concrete example of a natural class of finite diensional algebras, outside of a 2-Calabi Yau setting, with a mutation theory encoding the combinatorics of skew-symmetrisable cluster algebras. Importantly, we do not need to pass to the setting of preprojective algebras of Dynkin type at any point, and rather are able to develop this as a standalone mutation theory for mesh algebras of Dynkin type.
\end{remark}

\subsection{Quotient matrix}

The theory of taking quotient matrices (i.e.\ folding) presented in the following is classical. For background reading with a view towards cluster algebras, see for example Dupont's work on non-simply laced cluster algebras \cite{Dupont-non-simply-laced}, and the more general \cite{Dupont-ArXiv}.

\begin{definition}\label{D:matrix-gamma-action}
	Let $A = (\tilde{a}_{\tilde{M}\tilde{N}})_{\tilde{M},\tilde{N} \in \ind(T)}$ be a matrix with rows and columns labelled by $\ind(T)$. Then we say that $\gamma$ acts on $A$ if for all $\tilde{M},\tilde{N} \in \ind(T)$ we have $\tilde{a}_{{}_\gamma \tilde{M}{}_\gamma \tilde{N}} = \tilde{a}_{\tilde{M}\tilde{N}}$.
\end{definition}

\begin{definition}\label{D:quotient matrix}
	Let $\gamma$ act on the matrix $A = (a_{\tilde{M}\tilde{N}})_{\tilde{M},\tilde{N} \in \ind(T)}$. The {\em quotient matrix of $A$ by $\langle \gamma \rangle$} is the matrix
	$A/\langle \gamma \rangle = (a_{MN})_{M,N \in \min_\gamma(T)}$ with
	\[
		a_{MN} = \sum_{\tilde{M} \in \ind(M)} \tilde{a}_{\tilde{M}\tilde{N}},
	\]
	for some $\tilde{N} \in \ind(N)$.
\end{definition}

\begin{remark}
Note that the entry $a_{MN}$ of $A/\langle \gamma \rangle$ is independent of the choice of $\tilde{N}$ in Definition \ref{D:quotient matrix}. 
\end{remark}

\subsection{Matrix mutation}

We will relate the mutation of cluster tilting modules over $\La$ to the combinatorial process of matrix mutation, as introduced in the context of cluster algebras by Fomin and Zelevinsky \cite{Fomin-Zelevinsky-cluster-algebras-I}.

\begin{definition}[{\cite[Definition~4.2]{Fomin-Zelevinsky-cluster-algebras-I}}]
	Let $A = (a_{ij})_{1 \leq i,j \leq n}$ be a $(n \times n)$-matrix over $\mathbb{Z}$, and let $1 \leq k \leq n$. The mutation of $A$ in direction $k$ is defined to be the matrix $\mu_k(A) = (a'_{ij})_{1 \leq i,j \leq n}$ such that
	\[
		a'_{ij} = 	\begin{cases}
					-a_{ij} & \text{if $i =k$ or $j = k$} \\
					a_{ij} + \frac{|a_{ik}|a_{kj} + a_{ik}|a_{kj}|}{2} & \text{else}.
				\end{cases}
	\]
\end{definition}

If $A$ has only zeroes on its diagonal, we can express the mutation of $A$ in terms of matrix multiplication as follows. Denote by $I_n$ the $(n \times n)$ identity matrix, and, for $1 \leq i,j \leq n$, by $E_{ij}$ the matrix with $1$ in position $(i,j)$ and $0$s everywhere else. We set
\begin{align*}
	U_{A,k} & = 	I_n - 2E_{kk} + \sum_{a_{kj}\leq0}|a_{kj}|E_{kj}, 
	\\
	W_{A,k} & = I_n - 2E_{kk} + \sum_{a_{ik}\geq0}|a_{ik}|E_{ik}
\end{align*}
The matrices $U_{A,k} = (u_{ij})_{1 \leq i,j \leq n}$ and 
$W_{A,k}=(w_{ij})_{1 \leq i,j \leq n}$ have entries
\[
	u_{ij} = 	\begin{cases}
				-\delta_{ij} - a_{ij} & \text{if $i = k$ and $a_{ij} \leq 0$}\\
				\delta_{ij} & \text{else}
			\end{cases} \; \; 
	w_{ij} = 	\begin{cases}
				-\delta_{ij} + a_{ij} & \text{if $j = k$ and $a_{ji} \geq 0$}\\
				\delta_{ij} & \text{else}.
			\end{cases}	
\]

Note that if $A$ is skew-symmetric, then 
$U_{A,k} = W^t_{A,k}$, where for a matrix $B$, we denote by $B^t$ the transpose of $B$. The following description of matrix mutation is used by Berenstein, Fomin and Zelevinsky in the proof of \cite[Lemma~3.2]{Berenstein-Fomin-Zelevinsky}.

\begin{lemma}
\label{L:matrix mutation}
	Let $A = (a_{ij})_{1 \leq i,j \leq n}$ be an $(n \times n)$-matrix over $\mathbb{Z}$ with $a_{ii} = 0$ for all $1 \leq i \leq n$. Let $1 \leq k \leq n$ and set $U=U_{A,k}$ and $W=W_{A,k}$. Then we have
	\[
		\mu_k(A) = WAU.
	\]
\end{lemma}
Note that Lemma \ref{L:matrix mutation} in particular holds for skew-symmetrisable matrices.
We include a proof for the convenience of the reader, and to showcase that it holds in the generality stated.
\begin{proof}
	We compute $A' = WAU$ with entries $A' = (a'_{ij})_{1 \leq i,j \leq n}$. We have
$a'_{ij} = \sum_{1 \leq p,q \leq n}w_{ip}a_{pq}u_{qj}$.
	Note that $w_{ip} \neq 0$ implies $p = k$ or $p = i$, and symmetrically, $u_{qi} \neq 0$ implies $q = k$ or $q = i$, thus in the above sum, the index $p$ runs over $\{i,k\}$ and $q$ runs over $\{j,k\}$. Furthermore, since $A$ is assumed to have zeroes on the diagonal, we have $a_{kk} = 0$. It follows that
	\[
		a'_{ij} = \sum_{p \in \{i,k\}} \sum_{q \in \{j,k\}} w_{ip}a_{pq}u_{qj},
	\]
	and
	if $i = k$ and $j = k$, this yields
		$a'_{ij} = w_{kk}a_{kk}u_{kk} = 0 = -a_{kk} = -a_{ij}$.
	If $i = k$ and $j \neq k$, we obtain
		$a'_{ij} = 	 w_{kk}a_{kj}u_{jj} + w_{kk}a_{kk}u_{kj} = -a_{kj} = - a_{ij}$
	and symmetrically, if $i \neq k$ and $j = k$, we obtain
		$a'_{ij} = 	w_{ii}a_{ik}u_{kk} + w_{ik}a_{kk}u_{kk} = -a_{ik} = -a_{ij}$. 
	Finally, if $i \neq k$ and $j \neq k$, we have
	\[
		a'_{ij} = w_{ii}a_{ij}u_{jj} + w_{ii}a_{ik}u_{kj} + w_{ik}a_{kj}u_{jj} + w_{ik}a_{kk}u_{kj} = a_{ij} + a_{ik}u_{kj} + w_{ik}a_{kj}. 						
	\]
	A case distinction shows
	\begin{align*}
		a'_{ij} &= 	\begin{cases}
					a_{ij} + a_{ik}|a_{kj}| + |a_{ik}|a_{kj} = a_{ij}& \text{if $a_{kj} \leq 0$, $a_{ik} \geq 0$}\\
					a_{ij} + a_{ik}|a_{kj}|& \text{if $a_{kj} \leq 0$, $a_{ik} < 0$}\\
					a_{ij} + |a_{ik}|a_{kj} & \text{if $a_{kj} > 0$, $a_{ik} \geq 0$}\\
					a_{ij} & \text{if $a_{kj} > 0$, $a_{ik} < 0$}.
				\end{cases} &
			&= a_{ij} + \frac{|a_{ik}|a_{kj} + a_{ik}|a_{kj}|}{2}.
	\end{align*}	
	Comparing entry-wise, we observe that $A' = \mu_{k}(A)$, as desired.
\end{proof}

\subsection{The exchange matrix of $T$}

Consider the adjacency matrix $\tilde{B}_T$ of the quiver $\tilde{Q}_T$ of $E = \End_\La(T)$. We have $\tilde{B}_T = (\tilde{b}_{\tilde{M}\tilde{N}})_{\tilde{M},\tilde{N} \in \ind(T)}$ where
	\begin{align*}
	\tilde{b}_{\tilde{M}\tilde{N}} 	& =  \dim \Ext^1_E(S_{\tilde{N}},S_{\tilde{M}}) - \dim \Ext^1_E(S_{\tilde{M}}, S_{\tilde{N}}) \\
							& =  |\{\text{arrows $\tilde{N} \to \tilde{M}$ in $\tilde{Q}_T $}\}| -  |\{\text{arrows $\tilde{M} \to \tilde{N}$ in $\tilde{Q}_T $}\}|.
	\end{align*}
	Note that here we identify the vertices in the quiver $\tilde{Q}_T$ with the indecomposable summands of $T$.
Then $\gamma$ acts on $\tilde{B}_T$, and we call $B_T = \tilde{B}_T/\langle \gamma \rangle$ the {\em exchange matrix of $T$}. In line with the notation from \cite{GLS-3} we denote by $B^\circ_T$ its {\em principal part}, i.e.\ the submatrix of $B_T$ of columns and rows labelled by non-projective modules: 
\[
	B^\circ_T = (b_{MN})_{M,N \in \min_\gamma(T) \setminus \mathrm{proj}_\gamma(\La)},
\]
where $\mathrm{proj}_\gamma(\La)$ denotes the set of minimal $\gamma$-equivariant projective $\La$-modules.

\begin{remark}
	Note that by Item (1) of Theorem \ref{T:maximal rigid = 2-ct} the exchange  matrix $B_T$ of $T$, and its principal part $B^\circ_T$, have no non-zero entries on their diagonals, i.e.\ $b_{MM} = 0$ for all $M \in \min_\gamma(T)$.
\end{remark}

\begin{definition}\label{D:admissible}
	We say that the $\gamma$-equivariant cluster tilting module $T$ is {\em admissible} if for all %(non-projective would be enough)
	$\tilde{M},\tilde{N} \in \ind(T)$ and all $i \in \mathbb{Z}$ we have
	\[
		\tilde{b}_{\tilde{M}\tilde{N}} \geq 0 \Rightarrow \tilde{b}_{\tilde{M}{}_{\gamma^i}{\tilde{N}}} \geq 0, \; \text{and} \; \tilde{b}_{\tilde{M}\tilde{N}} \leq 0 \Rightarrow \tilde{b}_{\tilde{M}{}_{\gamma^i}{\tilde{N}}} \leq 0,
	\]
\end{definition}

\begin{remark}
	By \cite[Lemma~2.5]{Dupont-ArXiv} and Item (1) of Theorem \ref{T:maximal rigid = 2-ct}, if $T$ is an admissible $\gamma$-equivariant cluster tilting module, then $B_T$ is skew-symmetrisable.
\end{remark}

We are now ready to state the main result of this section.

\begin{thm}\label{T:mutation matrix}
	Let $T$ be an admissible $\gamma$-equivariant cluster tilting module
	and let $X \in \min_\gamma(T)$ be a non-projective minimal $\gamma$-equivariant summand of $T$. 
	Then
	\[
		B^\circ_{\mu_X(T)} = \mu_X(B^\circ_T).
	\]
\end{thm}

While we do need the assumption for $T$ to be admissible in the statement of Theorem \ref{T:mutation matrix}, the majority of the theory presented below holds for all cluster tilting modules. In the following, we thus do not assume $T$ to be admissible, unless explicitly stated. Our approach for proving Theorem \ref{T:mutation matrix} is inspired by \cite[Section~7]{GLS-3}.

\subsection{Ringel form and Cartan matrix}

In order to prove Theorem \ref{T:mutation matrix}, rather than  with the matrix $B$, we work with the  closely related Ringel form of ${\rm End}(T)$. This form is controlled
by the Cartan matrix of the algebra, and the change of the Cartan matrix under mutation can be seen directly from exchange sequences. 

Let $A$ be any finite dimensional algebra of finite global dimension. We recall the definition and crucial properties of the Ringel form. 
Let $P_1, \ldots, P_n$ be a complete set for the isomorphism classes of indecomposable projective $A$-modules. The Ringel-form is the bilinear form
\[
	\langle -, - \rangle \colon \mathbb{Z}^n \times \mathbb{Z}^n \to \mathbb{Z},
\]
defined via
\[
	\langle \underline{\dim}(M), \underline{\dim}(N) \rangle = \sum_{i = 0}^\infty \dim \Ext^i_A(M,N),
\] 
for any $M,N \in \modu \La$. Its matrix is given by $\tilde{R} = (\tilde{r}_{ij})_{1 \leq i,j \leq n}$, where
\[
	\tilde{r}_{ij} = \langle S_i, S_j \rangle,
\]
where $S_i$ denotes the simple top of $P_i$. The Cartan matrix of $A$ is the matrix $\tilde{C} = (\tilde{c}_{ij})_{1 \leq i,j \leq n}$ with 
$\tilde{c}_{ij} = \dim \Hom_A(P_i, P_j)$, it is invertible for $A$ of finite global dimension. Since $A$ is assumed to have finite global dimension, a Lemma by Ringel \cite[Section~2.4]{Ringel-tame} shows that the matrix $\tilde{R}$ is the inverse transpose of the Cartan matrix of $A$,
\[
	\tilde{R} = \tilde{C}^{-t}.
\]

We now return to our setting, where the algebra is ${\rm End}(T)$. We denote by $\tilde{R}_T$ the matrix of the Ringel form of $\End_\La(T)$, and by $\tilde{C}_T$ its Cartan matrix. The entries of the Cartan matrix $\tilde{C}_T = (\tilde{c}_{\tilde{M}\tilde{N}})_{\tilde{M}, \tilde{N} \in \ind(T)}$ are given by
\[
	\tilde{c}_{\tilde{M}\tilde{N}} = \dim \Hom_{\End_\La(T)}(\Hom_\La(\tilde{M},T), \Hom_\La(\tilde{N},T)) = \dim \Hom_\La(\tilde{N},\tilde{M}).
\] 
The automorphism $\gamma$ acts on $\tilde{C}$, and we denote by $C = (c_{MN})_{M,N \in \min_\gamma(T)}$ the quotient matrix $C = \tilde{C}/\langle \gamma \rangle$ with entries given by
\[
	c_{MN} = \sum_{\tilde{M} \in \ind(M)} \dim \Hom_\La(\tilde{N},\tilde{M}) = \dim \Hom_\La(\tilde{N},M),
\]
for any $\tilde{N} \in \ind(N)$. Furthermore, the automorphism $\gamma$ acts on the matrix $\tilde{R}$, and we consider the quotient matrix $R_T = \tilde{R}_T/\langle \gamma \rangle$. Furthermore we set
\[
	\tilde{G}_T = \tilde{C}_T^{t}
\]
to be the transpose of $\tilde{C}_T$ by $\gamma$, and denote by $G_T = \tilde{G}_T/\langle \gamma \rangle$ its quotient by $\gamma$. 
\begin{lemma}\label{L:Ringel and Cartan}
	Let $n$ be the number of minimal $\gamma$-equivariant summands of $T$. We have 
	\[
		R_T G_T = I_n,
	\]
	where $I_n$ denotes the $(n \times n)$-identity matrix.
\end{lemma}

\begin{proof}
	By \cite[Section~2.4]{Ringel-tame} we have $\tilde{R}_T = \tilde{C}_T^{-t} = \tilde{G}_T^{-1}$. The automorphism $\gamma$ acts on $\tilde{R}_T$, and on $\tilde{G}_T$, and trivially on their product $\tilde{R}_T\tilde{G}_T = I_n$. Therefore taking the quotient matrices, by Lemma \ref{L:matrix multiplication} below, we obtain
	\[
		R_TG_T = I_n.
	\]
\end{proof}

\begin{lemma}\label{L:matrix multiplication}
	Let $\gamma$ act on the square matrices $\tilde{A}, \tilde{B}$ and $\tilde{C} = \tilde{A}\tilde{B}$, with rows and columns labelled by $\ind(T)$. Then setting $A = \tilde{A}/\langle \gamma \rangle$, $B = \tilde{B}/\langle \gamma \rangle$ and $C = \tilde{C}/\langle \gamma \rangle$ we have
	\[
		AB = C.
	\]
\end{lemma}

\begin{proof}
	This is a straightforward calculation which we include for the convenience of the reader. For $F \in \{A,B,C\}$ we write $F = (f_{MN})_{M,N \in \min_\gamma(T)}$ and $\tilde{F} = (\tilde{f}_{\tilde{M}\tilde{N}})_{\tilde{M},\tilde{N} \in \ind(T)}$. 
	For each $M \in \min_\gamma(T)$ we fix a representative indecomposable summand $\tilde{M} \in \ind(M)$. Then the $(M,N)$th entry of the product $AB$ is given by
	
	\begin{align*}
		 (AB)_{MN} &= \sum_{K \in \min_\gamma(T)} a_{MK}b_{KN}  = \sum_{K \in \min_\gamma(T)} \sum_{p = 0}^{|M|-1}\tilde{a}_{{}_{\gamma^p}\tilde{M}\tilde{K}} \sum_{q = 0}^{|K|-1}\tilde{b}_{{}_{\gamma^q}\tilde{K}\tilde{N}} \\
		& = \sum_{K \in \min_\gamma(T)} \sum_{q = 0}^{|K|-1} \sum_{p = 0}^{|M|-1}\tilde{a}_{{}_{\gamma^{p+q}}\tilde{M}{}_{\gamma^q}\tilde{K}} \tilde{b}_{{}_{\gamma^q}\tilde{K}\tilde{N}} 
	\end{align*}
Note that for all $q \in \mathbb{Z}$ we have $\{{}_{\gamma^q}\tilde{M},{}_{\gamma^{q+1}}\tilde{M}, \ldots, {}_{\gamma^{q+|M|-1}}\tilde{M}\} = \ind(M)$ and so
	\begin{align*}
		 (AB)_{MN}  &= \sum_{K \in \min_\gamma(T)} \sum_{q = 0}^{|K|-1} \sum_{p = 0}^{|M|-1}\tilde{a}_{{}_{\gamma^{q+p}}\tilde{M}{}_{\gamma^q}\tilde{K}} \tilde{b}_{{}_{\gamma^q}\tilde{K}\tilde{N}}  &
		 		&= \sum_{\tilde{M}' \in \ind(M)} \sum_{K \in \min_\gamma(K)} \sum_{q = 0}^{|K|-1}\tilde{a}_{\tilde{M}'{}_{\gamma^q}\tilde{K}} \tilde{b}_{{}_{\gamma^q}\tilde{K}\tilde{N}} \\
				 &=  \sum_{\tilde{M}' \in \ind(M)} \sum_{\tilde{K}' \in \ind(T)} \tilde{a}_{\tilde{M}'\tilde{K}'} \tilde{b}_{\tilde{K}'\tilde{N}} &
				 &=  \sum_{\tilde{M}' \in \ind(M)} \tilde{c}_{\tilde{M}'\tilde{N}} = c_{MN}.
	\end{align*}

\end{proof}

\subsection{Ringel form and exchange matrix}

In order to better understand the matrix $\tilde{R}_T$ and its quotient $R_T$, with the goal to relate it to the exchange matrix $B_T$ we make the following observation.

\begin{proposition}\label{P:Ext3-i}
Let $\tilde{X}$ be an indecomposable non-projective summand of $T$. Then for all $0 \leq i \leq 3$ and for any simple $E = \End_\La(T)$-module $S_{\tilde{Z}}$ we have
\[
	\dim \Ext^{3-i}_E(S_{\tilde{X}},S_{\tilde{Z}}) = \dim \Ext^i_E(S_{\tilde{Z}},S_{{}_{\gamma^{-1}}\tilde{X}}).
\]
\end{proposition}

\begin{proof}
Consider the projective resolution of $S_{\tilde{X}}$ from Sequence (\ref{S:resolution of SX}) in the proof of Proposition \ref{P:gldim}
\begin{align}
	0 \to \Hom_\La({}_{\gamma^{-1}}\tilde{X},T) \to \Hom_\La(\tilde{T}'',T) \to \Hom_\La(\tilde{T}',T) \to \Hom_\La(\tilde{X},T) \to S_{\tilde{X}} \to 0,
\end{align}
that we obtained from applying $\Hom_\La(-,T)$ to appropriate summands of the exchange sequences (\ref{S:forwards}) and (\ref{S:backwards}) from Remark \ref{R:twist} and combining the obtained exact sequences. Based on this we set $P_0 = P_{\tilde{X}}$, $P_1 = \Hom_\La(\tilde{T}',T)$, $P_2 = \Hom_\La(\tilde{T}'',T)$ and $P_3 = P_{{}_{\gamma^{-1}}\tilde{X}}$.

Analogously, by applying $\Hom_\La(T,-)$ to appropriate summands of the sequence and postcomposing $D(-) = \Hom_\mathbb{K}(-,\mathbb{K})$, and then combining the obtained exact sequences, we obtain an injective resolution of $S_{{}_{\gamma^{-1}}\tilde{X}}$:
\[
	0 \to S_{{}_{\gamma^{-1}}\tilde{X}} \to D\Hom_\La(T,{}_{\gamma^{-1}}\tilde{X}) \to D\Hom_\La(T,\tilde{T}'') \to D\Hom_\La(T,\tilde{T}') \to D\Hom_\La(T, \tilde{X}) \to 0.
\]
Based on this we set $I_0 = I_{{}_{\gamma^{-1}}\tilde{X}}$, $I_1 =  D\Hom_\La(T,\tilde{T}'')$, $I_2 = D\Hom_\La(T,\tilde{T}')$ and $I_3 = I_{\tilde{X}}$.

For $E$-modules $M,N$ we denote by $(M : N)$ the multiplicity of $N$ as a summand of $M$. Then for every indecomposable summand $\tilde{Z}$ of $T$ and all $0 \leq i \leq 3$ we have
\begin{align*}
	(P_i : P_{\tilde{Z}}) = \dim \Ext^i_E(S_{\tilde{X}}, S_{\tilde{Z}})\; \text{and}  \; (I_i : I_{\tilde{Z}}) = \dim \Ext^i_E(S_{\tilde{Z}}, S_{{}_{\gamma^{-1}}\tilde{X}}).
\end{align*}
By $\mathbb{K}$-duality we see that for all $0 \leq i \leq 3$ we have $(P_{3-i} : P_{\tilde{Z}}) = (I_i : I_{\tilde{Z}})$ and the claim follows.

\end{proof}

We denote by $R^\circ_T$ the principal part of the matrix $R_T$, i.e.\ the submatrix labelled by non-projective minimal $\gamma$-equivariant modules:
\[
	R^\circ_T = (r_{MN})_{M,N \in \min_\gamma(T) \setminus \mathrm{proj}_\gamma(\La)}.
\]

\begin{lemma}\label{L:exchange and Ringel}
	The principal parts of $B_T$ and $R_T$ coincide:
	\[
		B^\circ_T = R^\circ_T.
	\]
\end{lemma}

\begin{proof}
	Let $B_T = (b_{MN})_{M,N \in \min_\gamma(T)}$ and $R_T = (r_{MN})_{M,N \in \min_\gamma(T)}$. Let $M$ and $N$ be non-projective $\gamma$-equivariant summands of $T$, and let $\tilde{N}$ be an indecomposable summand of $N$.  Set $E = \End_\La(T)$. By Proposition \ref{P:Ext3-i} and Item (1) from Theorem \ref{T:maximal rigid = 2-ct} for all $\tilde{M} \in \ind(M)$ we have
\begin{align*}
	\tilde{r}_{\tilde{M}\tilde{N}} & = 	\begin{cases}
								\dim \Ext_{E}^1(S_{\tilde{N}}, S_{{}_{\gamma}\tilde{M}}) - \dim \Ext_{E}^1(S_{\tilde{M}},S_{\tilde{N}}) & \text{if $M \ncong N$}\\
								\dim \Hom_{E}(S_{\tilde{M}},S_{\tilde{N}}) - \dim \Hom_{E}(S_{\tilde{N}},S_{{}_{\gamma}\tilde{M}}) & \text{if $M \cong N$}.
							\end{cases}\\
\end{align*}
If $M \ncong N$ we compute
	\begin{align*}
		b_{MN} &= \sum_{\tilde{M} \in \ind(M)} \tilde{b}_{\tilde{M}\tilde{N}}  = \sum_{\tilde{M} \in \ind(M)}\dim \Ext^1_E(S_{\tilde{N}},S_{\tilde{M}}) - \sum_{\tilde{M} \in \ind(M)}\dim \Ext^1_E(S_{\tilde{M}}, S_{\tilde{N}})\\
		& = \sum_{\tilde{M} \in \ind(M)}\dim \Ext^1_E(S_{\tilde{N}},S_{{}_{\gamma}\tilde{M}}) - \sum_{\tilde{M} \in \ind(M)}\dim \Ext^1_E(S_{\tilde{M}}, S_{\tilde{N}}) = \sum_{\tilde{M} \in \ind(M)} \tilde{r}_{\tilde{M}\tilde{N}} = r_{MN},
	\end{align*}
	and if $M \cong N$ by Item (1) from Theorem \ref{T:maximal rigid = 2-ct} we see that
	\begin{align*}
		b_{NN} &= \sum_{\tilde{M} \in \ind(N)} \tilde{b}_{\tilde{M}\tilde{N}}  = \sum_{\tilde{M} \in \ind(N)}\dim \Ext^1_E(S_{\tilde{N}},S_{\tilde{M}}) - \sum_{\tilde{M} \in \ind(N)}\dim \Ext^1_E(S_{\tilde{M}}, S_{\tilde{N}}) = 0 \\
		& =  \sum_{\tilde{M} \in \ind(N)}\dim \Hom_E(S_{\tilde{M}}, S_{\tilde{N}}) - \sum_{\tilde{M} \in \ind(N)}\dim \Hom_E(S_{\tilde{N}},S_{{}_{\gamma}\tilde{M}}) = \sum_{\tilde{M} \in \ind(N)} \tilde{r}_{\tilde{M}\tilde{N}} = r_{NN},
	\end{align*}
\end{proof}

\subsection{Proof of Theorem \ref{T:mutation matrix}}

Let $X$ be a minimal $\gamma$-equivariant summand of our $\gamma$-equivariant cluster tilting module $T$. We consider the mutation
\[
	T^* = \mu_X(T) = T/X \oplus Y
\]
of $T$ in the direction of $X$. 
Before we prove Theorem \ref{T:mutation matrix} we observe the behaviour of the matrix $G_T$ under mutation of $T$ in the direction of $X$. Recall that $G_T$ is the quotient of the transpose of the Cartan matrix $\tilde{C}_T$ by $\gamma$; $G_T = \tilde{C}^t/\langle \gamma \rangle$. Due to its description in terms of the Cartan matrix, the new matrix $G_{T^*}$ can be determined in a straightforward way using representation theory.

As before, let $B_T = (b_{MN})_{M,N \in \min_{\gamma}(T)}$ be the exchange matrix of $T$, the quotient of the matrix $\tilde{B} = (\tilde{b}_{\tilde{M}\tilde{N}})_{\tilde{M},\tilde{N} \in \ind(T)}$.
We set
\begin{align*}
	U_{T,X} & = 	I_n - 2E_{XX} + \sum_{b_{XZ}\leq0}|b_{XZ}|E_{XZ}, \\ 
	W_{T,X} & = I_n - 2E_{XX} + \sum_{b_{ZX}\geq0}|b_{ZX}|E_{ZX}.
\end{align*}
Note that these are precisely the matrices describing matrix mutation of $B_T$ in direction $X$, as in Lemma \ref{L:matrix mutation}.

\begin{proposition}\label{P:mutation Cartan}
	Assume $T$ is an admissible $\gamma$-equivariant cluster tilting module. Let $T^* = \mu_X(T)$ be the mutation of $T$ in the direction of $X$, and set $U = U_{T,X}$ and 
	$W = W_{T,X}$. Then
	\[
		G_{T^*} = UG_TW.
	\]
\end{proposition}

Note that in Proposition \ref{P:mutation Cartan}, we explicitly need $T$ to be an admissible cluster tilting module, see Section \ref{S:admissible} for details on admissible cluster tilting modules in $\modu \La$.

\begin{proof}
	We compare the matrices $C^* = C_{T^*}$ and 
	$C' = UC_TV^t$ component wise.  	
	{\bf Step 1:}
	First, consider the matrix $G' = (g'_{MN})_{M,N \in \ind(T)}$. We have $U = (u_{MN})_{M,N \in \min_\gamma(T)}$ and
	$W = (w_{MN})_{M,N \in \min_\gamma(T)}$ and
	\[
		g'_{MN} = \sum_{L,K \in \ind(T)}u_{ML}g_{LK}w_{KN}.
	\]
We compute
	\[
		g'_{MN} = \begin{cases}
					g_{MN} & \text{if $M \neq X$, $N \neq X$} \\
					-g_{MX} + \sum_{b_{KX } \geq 0}b_{KX}g_{MK} & \text{if $M \neq X, N = X$}\\
					-g_{XN} - \sum_{b_{XL \leq 0}}b_{XL}g_{LN} & \text{if $M = X, N \neq X$}\\
					g_{XX} - \sum_{b_{LX} \geq 0}b_{LX}g_{XL} + \sum_{b_{XK \leq 0}}b_{XK}g_{KX}  - \sum_{b_{XK} \leq 0} \sum_{b_{LX} \geq 0}  b_{XK}g_{KL}b_{LX}& \text{if $M = N = X$}.
				\end{cases}
	\]

{\bf Step 2:}	
We use representation theory to express the Cartan matrix $\tilde{C}_{T^*}$ of $\End_\La(T^*)$ in terms of the Cartan matrix $\tilde{C}_{T} = (\tilde{c}_{\tilde{M}\tilde{N}})_{\tilde{M},\tilde{N} \in \ind(T)}$ of $\End_\La(T)$, and, as a consequence, express its transpose matrix $\tilde{G}_{T^*} = \tilde{C}_{T^*}^t$ in terms of $\tilde{G}_T = \tilde{C}_T^t$. We have $\tilde{C}_{T^*} = (\tilde{c}^*_{\tilde{M}\tilde{N}})_{\tilde{M},\tilde{N} \in \ind(T^*)}$, where
$\tilde{c}^*_{\tilde{M}\tilde{N}} = \dim \Hom_\La(\tilde{N},\tilde{M})$.
Let now $\tilde{M},\tilde{N} \in \ind(T^*)$. We immediately see that if $\tilde{M}, \tilde{N} \in \ind(T^*) \setminus \ind(Y) = \ind(T) \setminus \ind(X)$ then
\[
	\tilde{c}^*_{\tilde{M}\tilde{N}} = \tilde{c}_{\tilde{M}\tilde{N}}.
\]
We next want to express $\tilde{c}^*_{\tilde{M}\tilde{N}}$ in terms of entries in $\tilde{C}_T$ for the cases where $\tilde{M}$ or $\tilde{N}$ are in $\add(Y)$. Assume $X = \bigoplus_{i=0}^n {}_{\gamma^i}\tilde{X}$, where $\tilde{X}$ is indecomposable and consider the exchange sequence
\begin{equation}\label{E:exchange sequence}
	0 \to X \to T' \to Y \to 0 = \bigoplus_{i=0}^n 0 \to {}_{\gamma^i}\tilde{X} \xrightarrow{{}_{\gamma^i}\tilde{f}} {}_{\gamma^i}\tilde{T}' \xrightarrow{{}_{\gamma^i}\tilde{g}} {}_{\gamma^i}\tilde{Y} \to 0.
\end{equation}
For $\La$-modules $L,K$ we denote by $(L : K)$ the multiplicity of $K$ as a summand of $L$. Since ${}_{\gamma^i}\tilde{f}$, respectively ${}_{\gamma^i}\tilde{g}$, are a minimal left, respectively right, $\add(T/X)$-approximations, for an indecomposable $\La$-module $\tilde{Z} \in \add(T/X)$ we have
\[
	(_{\gamma^i} \tilde{T}' : \tilde{Z}) = \dim \Ext^1_{\End_\La(T)}(_{\gamma^i}\tilde{X},\tilde{Z}) = 
		\begin{cases}
			\tilde{b}_{\tilde{Z}{}_{\gamma^i}\tilde{X}} & \text{if $\tilde{b}_{\tilde{Z}{}_{\gamma^i}\tilde{X}} \geq 0$}\\
			0 & \text{else},
		\end{cases}
\]
that is 
\[
	{}_{\gamma^i}\tilde{T}'  = \bigoplus_{\tilde{b}_{\tilde{Z}{}_{\gamma^i}\tilde{X}} \geq 0} \tilde{Z}^{\tilde{b}_{\tilde{Z}{}_{\gamma^i}\tilde{X}}}.
\]
Assume first that $\tilde{M} \in \ind(T^*) \setminus \ind(Y)$ and $\tilde{N} = {}_{\gamma^i}\tilde{Y} \in \ind(Y)$. Applying $\Hom_\La(-,\tilde{M})$ to the appropriate summand (featuring ${}_{\gamma^i}\tilde{Y}$) of Sequence (\ref{E:exchange sequence}) yields the exact sequence
\[
	 0 \to \Hom_\La({}_{\gamma^i}\tilde{Y},\tilde{M}) \xrightarrow{} \Hom(\bigoplus_{\tilde{b}_{\tilde{Z}{}_{\gamma^i}\tilde{X}} \geq 0} \tilde{Z}^{\tilde{b}_{\tilde{Z}{}_{\gamma^i}\tilde{X}}},\tilde{M}) \xrightarrow{} \Hom({}_{\gamma^i}\tilde{X},\tilde{M}) \to \Ext^1({}_{\gamma^i}\tilde{Y},\tilde{M}) = 0,
\]
and thus
\begin{align*}
	\tilde{c}^*_{\tilde{M}{}_{\gamma^i}\tilde{Y}} & = \dim \Hom_\La( {}_{\gamma^i}\tilde{Y} ,\tilde{M}) = - \dim \Hom_\La({}_{\gamma^i}\tilde{X},\tilde{M}) +\sum_{\tilde{b}_{\tilde{Z}{}_{\gamma^i}\tilde{X}} \geq 0} \tilde{b}_{\tilde{Z}{}_{\gamma^i}\tilde{X}}\dim \Hom_\La(\tilde{Z}, \tilde{M}) 
	\\ 
	&= - \tilde{c}_{\tilde{M}{}_{\gamma^i}\tilde{X}} + \sum_{\tilde{b}_{\tilde{Z}{}_{\gamma^i}\tilde{X}} \geq 0} \tilde{b}_{\tilde{Z}{}_{\gamma^i}\tilde{X}}\tilde{c}_{\tilde{M}\tilde{Z}}.
\end{align*}
Similarly, if $\tilde{M} = {}_{\gamma^i}\tilde{Y} \in \ind(Y)$ and $\tilde{N} \in \ind(T^*) \setminus \ind(Y)$, applying $\Hom_\La(\tilde{N},-)$ to the appropriate summand (featuring ${}_{\gamma^i}\tilde{Y}$) of Sequence (\ref{E:exchange sequence}) yields 
\begin{align*}
	\tilde{c}^*_{{}_{\gamma^i}\tilde{Y}\tilde{N}} 
	& = - \tilde{c}_{{}_{\gamma^i}\tilde{X}\tilde{N}} + \sum_{\tilde{b}_{\tilde{Z}{}_{\gamma^i}\tilde{X}} \geq 0} \tilde{b}_{\tilde{Z}{}_{\gamma^i}\tilde{X}}\tilde{c}_{\tilde{Z}\tilde{N}}.
\end{align*}
Finally, assume $\tilde{M}, \tilde{N} \in \add(Y)$, with $\tilde{M} = {}_{\gamma^i}\tilde{Y}$ and $\tilde{N} = {}_{\gamma^j}\tilde{Y}$. Applying $\Hom_\La(-,{}_{\gamma^i}\tilde{Y})$ to the appropriate summand (featuring ${}_{\gamma^j}\tilde{Y}$) of Equation (\ref{E:exchange sequence}) yields
\begin{align*}
	\tilde{c}^*_{{}_{\gamma^i}\tilde{Y}{}_{\gamma^j}\tilde{Y}} & = \dim \Hom_\La( {}_{\gamma^j}\tilde{Y} ,{}_{\gamma^i}\tilde{Y}) = - \dim \Hom_\La({}_{\gamma^j}\tilde{X},{}_{\gamma^i}\tilde{Y}) + \sum_{\tilde{b}_{\tilde{Z}{}_{\gamma^j}\tilde{X}} \geq 0} \tilde{b}_{\tilde{Z}{}_{\gamma^j}\tilde{X}}\dim \Hom_\La(\tilde{Z}, {}_{\gamma^i}\tilde{Y})
\end{align*}
In order to break down the right hand side of the above equation, first apply $\Hom_\La({}_{\gamma^j}\tilde{X},-)$ to the appropriate summand (featuring ${}_{\gamma^i}\tilde{Y}$) of Sequenc (\ref{E:exchange sequence}) to obtain
\begin{align*}
	\dim \Hom_\La({}_{\gamma^j}\tilde{X},{}_{\gamma^i}\tilde{Y}) & = - \dim \Hom_\La({}_{\gamma^j}\tilde{X},{}_{\gamma^i}\tilde{X}) + \sum_{\tilde{b}_{\tilde{Z}{}_{\gamma^i}\tilde{X}} \geq 0} \tilde{b}_{\tilde{Z}{}_{\gamma^i}\tilde{X}} \dim \Hom_\La({}_{\gamma^j}\tilde{X}, \tilde{Z}) \\
	& = - \tilde{c}_{{}_{\gamma^i}\tilde{X}{}_{\gamma^j}\tilde{X}} +  \sum_{\tilde{b}_{\tilde{Z}{}_{\gamma^i}\tilde{X}} \geq 0} \tilde{b}_{\tilde{Z}{}_{\gamma^i}\tilde{X}} \tilde{c}_{\tilde{Z}{}_{\gamma^j}\tilde{X}}.
\end{align*}
Second, apply $\Hom_\La(\tilde{Z},-)$, for $\tilde{Z} \in \ind(T/X)$ to the appropriate summand of Sequence (\ref{E:exchange sequence}) to obtain
\[
	\dim \Hom_\La(\tilde{Z}, {}_{\gamma^i}\tilde{Y})  =  - \tilde{c}_{{}_{\gamma^i}\tilde{X}\tilde{Z}} + \sum_{\tilde{b}_{\tilde{L}{}_{\gamma^i}\tilde{X}} \geq 0} \tilde{b}_{\tilde{L}{}_{\gamma^i}\tilde{X}}\tilde{c}_{\tilde{L}\tilde{Z}}.
\]
Combining these results, we get
\begin{align*}
	\tilde{c}^*_{{}_{\gamma^i}\tilde{Y}{}_{\gamma^j}\tilde{Y}}  
	& = 	\tilde{c}_{{}_{\gamma^i}\tilde{X}{}_{\gamma^j}\tilde{X}} -  
		\sum_{\tilde{b}_{\tilde{Z}{}_{\gamma^i}\tilde{X}} \geq 0} \tilde{b}_{\tilde{Z}{}_{\gamma^i}\tilde{X}} \tilde{c}_{\tilde{Z}{}_{\gamma^j}\tilde{X}} - 					\sum_{\tilde{b}_{\tilde{Z}{}_{\gamma^j}\tilde{X}} \geq 0}  \tilde{b}_{\tilde{Z}{}_{\gamma^j}\tilde{X}}\tilde{c}_{{}_{\gamma^i}\tilde{X}\tilde{Z}} + 					\sum_{\tilde{b}_{\tilde{Z}{}_{\gamma^j}\tilde{X}} \geq 0}  \sum_{\tilde{b}_{\tilde{L}{}_{\gamma^i}\tilde{X}} \geq 0} \tilde{b}_{\tilde{Z}{}_{\gamma^j}\tilde{X}} \tilde{b}_{\tilde{L}{}_{\gamma^i}\tilde{X}}\tilde{c}_{\tilde{L}\tilde{Z}}.
\end{align*}

To summarise, taking the transpose matrices $\tilde{G}_{T^*} = (\tilde{g}^*_{\tilde{M}\tilde{N}})_{\tilde{M},\tilde{N} \in \ind(T^*)}$ and $\tilde{G}_{T} = (\tilde{g}_{\tilde{M}\tilde{N}})_{\tilde{M},\tilde{N} \in \ind(T)}$ of $\tilde{C}_{T^*}$ and $\tilde{C}_T$ respectively we obtain
\[
	\tilde{g}^*_{\tilde{M}\tilde{N}} = 	\begin{cases}
								\tilde{g}_{\tilde{M}\tilde{N}}  & \text{if $\tilde{M},\tilde{N} \in \ind(T^*) \setminus \ind(Y)	= \ind(T) \setminus \ind(X)$}\\
								- \tilde{g}_{\tilde{M}{}_{\gamma^i}\tilde{X}} + \sum_{\tilde{b}_{\tilde{Z}{}_{\gamma^i}\tilde{X}} \geq 0} \tilde{b}_{\tilde{Z}{}_{\gamma^i}\tilde{X}}\tilde{g}_{\tilde{M}\tilde{Z}} & \text{if $\tilde{M} \in \ind(T^*) \setminus \ind(Y), \tilde{N} = {}_{\gamma^i}\tilde{Y}$ for $i \in \bZ$}\\
								 - \tilde{g}_{{}_{\gamma^i}\tilde{X}\tilde{N}} + \sum_{\tilde{b}_{\tilde{Z}{}_{\gamma^i}\tilde{X}} \geq 0} \tilde{b}_{\tilde{Z}{}_{\gamma^i}\tilde{X}}\tilde{g}_{\tilde{Z}\tilde{N}}			& \text{if $\tilde{M} = {}_{\gamma^i}\tilde{Y}$ for $i \in \bZ$, $\tilde{N} \in \ind(T^*) \setminus \ind(Y)$}					
							\\
							\tilde{g}_{{}_{\gamma^j}\tilde{X}{}_{\gamma^i}\tilde{X}} -  
		\sum_{\tilde{b}_{\tilde{Z}{}_{\gamma^i}\tilde{X}} \geq 0} \tilde{b}_{\tilde{Z}{}_{\gamma^i}\tilde{X}} \tilde{g}_{{}_{\gamma^j}\tilde{X}\tilde{Z}} \\ - 	\sum_{\tilde{b}_{\tilde{Z}{}_{\gamma^j}\tilde{X}} \geq 0}  \tilde{b}_{\tilde{Z}{}_{\gamma^j}\tilde{X}}\tilde{g}_{\tilde{Z}{}_{\gamma^i}\tilde{X}}  \\-					\sum_{\tilde{b}_{{}_{\gamma^j}\tilde{X}\tilde{Z}} \leq 0}  \sum_{\tilde{b}_{\tilde{L}{}_{\gamma^i}\tilde{X}} \geq 0} \tilde{b}_{{}_{\gamma^j}\tilde{X}\tilde{Z}}\tilde{g}_{\tilde{Z}\tilde{L}} \tilde{b}_{\tilde{L}{}_{\gamma^i}\tilde{X}} & \text{if $\tilde{M} = {}_{\gamma^j}\tilde{Y}, \tilde{N} = {}_{\gamma^i}\tilde{Y}$ for $i ,j \in \bZ$}
						 
	 \end{cases} 
\]
{\bf Step 3:}
Now we take quotients by $\langle \gamma\rangle$, in the setting of matrices, and compute the entries of the quotient matrix $G_{T^*} = \tilde{G}_{T^*}/\langle \gamma \rangle$ in terms of $G_T = \tilde{G}_T/\langle \gamma \rangle$. Note that, while we have not used the assumption that $T$ is admissible in Steps 1 or 2, the proof from here on relies on this assumption. Set $G_{T^*} = (g^*_{MN})_{M,N \in \min_\gamma)(T^*)}$ and $G_T = (g_{MN})_{M,N \in \min_\gamma(T)}$. 
We have
\[
	g^*_{MN} = \sum_{\tilde{M} \in \add(M)} \tilde{g}^*_{\tilde{M}\tilde{N}}.
\]
For $M,N \in \min_\gamma(T^*) \setminus \{Y\}$ we obtain $g^*_{MN} = g_{MN}$.
For $M \in \min_\gamma(T^*) \setminus \{Y\}$ and $N = Y$, we obtain
\begin{align*}
	g^*_{MY} & =  - \sum_{\tilde{M} \in \ind(M)} \tilde{g}_{\tilde{M}\tilde{X}} + \sum_{\tilde{M} \in \ind(M)} \sum_{\tilde{b}_{\tilde{Z}{}_{\gamma^i}\tilde{X}} \geq 0} \tilde{b}_{\tilde{Z}{}_{\gamma^i}\tilde{X}}  \tilde{g}_{\tilde{M}\tilde{Z}} 
			& = & - g_{MX} + \sum_{\tilde{b}_{\tilde{Z}{}_{\gamma^i}\tilde{X}} \geq 0}  \sum_{\tilde{M} \in \ind(M)} \tilde{g}_{\tilde{M}\tilde{Z}} \tilde{b}_{\tilde{Z}{}_{\gamma^i}\tilde{X}}  \\
			& =  - g_{MX} + \sum_{b_{ZX} \geq 0} \sum_{\tilde{Z} \in \ind(Z)} \sum_{\tilde{M} \in \ind(M)} \tilde{g}_{\tilde{M}\tilde{Z}} \tilde{b}_{\tilde{Z}{}_{\gamma^i}\tilde{X}} 
			& = & - g_{MX} + \sum_{b_{ZX} \geq 0} \sum_{\tilde{Z} \in \ind(Z)} g_{{M}{Z}} \tilde{b}_{\tilde{Z}{}_{\gamma^i}\tilde{X}} \\
			& = - g_{MX} + \sum_{b_{ZX} \geq 0} b_{ZX}g_{MZ}.
\end{align*}
A symmetric calculation for $M = Y$ and $N \in \min_\gamma(T^*) \setminus \{Y\}$, keeping in mind that the matrix $\tilde{B}_T$ is skew-symmetric, yields
\begin{align*}
	g^*_{YN} & = \sum_{\tilde{Y} \in \ind(Y)} \tilde{g}_{\tilde{Y}\tilde{N}} = \sum_{\tilde{X} \in \ind(X)} \big( - \tilde{g}_{{}_{\gamma^i}\tilde{X}\tilde{N}} - \sum_{\tilde{b}_{{}_{\gamma^i}\tilde{X}\tilde{Z}} \leq 0} \tilde{b}_{{}_{\gamma^i}\tilde{X}\tilde{Z}}\tilde{g}_{\tilde{Z}\tilde{N}} \big) %\\
	& = - g_{XN} - \sum_{b_{XZ} \leq 0} b_{XZ}g_{ZN}.
\end{align*}
Finally, if $M = N = Y$ we obtain
\begin{align*}
	g_{YY} & = \sum_{\tilde{X}' \in \ind(X)} \Big( \tilde{g}_{\tilde{X}'\tilde{X}} -  
		\sum_{\tilde{b}_{\tilde{Z}\tilde{X}} \geq 0} \tilde{b}_{\tilde{Z}\tilde{X}} \tilde{g}_{\tilde{X}'\tilde{Z}}  - 	\sum_{\tilde{b}_{\tilde{Z}\tilde{X}'} \geq 0}  \tilde{b}_{\tilde{Z}\tilde{X}'}\tilde{g}_{\tilde{Z}\tilde{X}}  -					\sum_{\tilde{b}_{\tilde{X}'\tilde{Z}} \leq 0}  \sum_{\tilde{b}_{\tilde{L}\tilde{X}} \geq 0} \tilde{b}_{\tilde{X}'\tilde{Z}}\tilde{g}_{\tilde{Z}\tilde{L}} \tilde{b}_{\tilde{L}\tilde{X}} \Big) \\
		&= g_{XX} - \sum_{b_{ZX} \geq 0} b_{ZX}g_{XZ} + \sum_{b_{XZ} \leq 0} b_{XZ}g_{ZX} - \sum_{b_{XZ} \leq 0} \sum_{b_{LX} \geq 0}  \sum_{\tilde{Z} \in \ind(Z)}  \sum_{\tilde{X}' \in \ind(X)}   \tilde{b}_{\tilde{X}'\tilde{Z}}\sum_{\tilde{L} \in \ind(L)}\tilde{g}_{\tilde{Z}\tilde{L}} \tilde{b}_{\tilde{L}\tilde{X}} \\
		&= g_{XX} - \sum_{b_{ZX} \geq 0} b_{ZX}g_{XZ} + \sum_{b_{XZ} \leq 0} b_{XZ}g_{ZX} - \sum_{b_{XZ} \leq 0} \sum_{b_{LX} \geq 0}  b_{XZ}  \sum_{\tilde{L} \in \ind(L)}  \sum_{\tilde{Z} \in \ind(Z)} \tilde{g}_{\tilde{Z}\tilde{L}} \tilde{b}_{\tilde{L}\tilde{X}} \\
		&= g_{XX} - \sum_{b_{ZX} \geq 0} b_{ZX}g_{XZ} + \sum_{b_{XZ} \leq 0} b_{XZ}g_{ZX} - \sum_{b_{XZ} \leq 0} \sum_{b_{LX} \geq 0}  b_{XZ}  g_{ZL} b_{LX}.
\end{align*}

To summarise, we have
\[
	g^*_{MN} = 	\begin{cases}
					g_{MN} & \text{if $M \neq Y, N \neq Y$}\\
					-g_{MX} + \sum_{b_{KX} \geq 0} b_{KX}g_{MK} & \text{if $M \neq Y, N = Y$}\\
					-g_{XN} - \sum_{b_{XL} \leq 0} b_{XL}g_{LN} & \text{if $M = Y, N \neq Y$} \\
					g_{XX} - \sum_{b_{KX} \geq 0} b_{KX}g_{XK} + \sum_{b_{XL} \leq 0}b_{XL}g_{LX} - \sum_{b_{XK} \leq 0} \sum_{b_{LX} \geq 0}  b_{XK} g_{KL} b_{LX} & \text{if $M = N = Y$.}
				\end{cases}
\]
Comparing with the entries of $G'$, bearing in mind that in $G^*$ the row and column labelled by $X$ in $G$ and $G'$ got relabelled by $Y$, we see that $G' = G^*$.
\end{proof}

We can now prove Theorem \ref{T:mutation matrix}.

\begin{proof}
	Set $U = U_{T,X} = (u_{MN})_{M,N \in \min_\gamma(T)}$ and 
	$W=W_{T,X}=(w_{MN})_{M,N \in \min_\gamma(T)}$. A straight-forward calculation shows that 
	$U^2 = W^2 = I_n$.
	Set $T^* = \mu_X(T)$. By Lemma \ref{L:Ringel and Cartan} and \ref{P:mutation Cartan} 
	we have
	\[
		R_{T^*} = G_{T^*}^{-1} = (UG_TW)^{-1} = WG_T^{-1}U = WR_TU.
	\]
	Set $U^\circ$ and  
	$W^\circ$ respectively to be the submatrices of $U$ and  
	$W$ consisting of rows and columns labelled by non-projective $\gamma$-equivariant summands of $T$. 	
	By Lemma \ref{L:exchange and Ringel} the matrices describing the mutation of $R^\circ_T$ in the direction of $X$ according to Lemma \ref{L:matrix mutation} are precisely $U^\circ$ and %$V^\circ$ 
	$W^\circ$ and we have
	\[
		\mu_X(R^{\circ}_T) = W^\circ R^\circ_T{U^\circ}.
	\]
Note that we have  
$u_{MN} = w_{NM} = 0$ whenever $M$ is projective and $N$ is non-projective, and let $n$ be the number of projective minimal $\gamma$-equivariant $\La$-modules. Thus, in block-matrix notation, where $*$ is a placeholder for any submatrix with potentially non-zero entries, we have
\[
	WR_TU = \begin{bmatrix} W^\circ & 0 \\ * & I_n \end{bmatrix}\begin{bmatrix} R_T^\circ & * \\ * & * \end{bmatrix}\begin{bmatrix} U^\circ & * \\ 0 & I_n \end{bmatrix} = \begin{bmatrix} W^\circ R_T^\circ U^\circ & * \\ * & * \end{bmatrix}.
\]
Therefore, denoting by $(WR_TU)^\circ$ the submatrix of $WR_TU$ with rows and columns labelled by non-projective minimal $\gamma$-equivariant summands of $T$, we have that
\[
	\mu_X(R^{\circ}_T) =W^\circ R^\circ_T{U^\circ}= (WR_TU)^\circ.
\]
	Therefore we conclude that
	\[
		B^\circ_{\mu_{X}(T)} = B^\circ_{T^*} = R^\circ_{T^*} = (WR_TU)^\circ = \mu_X(R^{\circ}_T) = \mu_X(B^{\circ}_T).
	\]
\end{proof}

\subsection{Admissibility of cluster tilting modules}\label{S:admissible}

Note that it is straightforward to calculate that the start module of $\modu \La$, as constructed in the proof of Theorem \ref{thm:La_2_ct}, is admissible for every mesh algebra $\La$. However, for $\La = P(\bF_4)$ and $\La = P(\bB_k)$ for $k \geq 2$ we do not know whether all cluster tilting modules (or even all cluster tilting modules that are related to the start module via mutation) are admissible. In particular, admissibility is not in general conserved under mutation, cf.\ \cite[Remark~2.18]{Dupont-ArXiv}.

However, in the remaining mesh algebras, all cluster tilting modules are admissible. This is clear for the simply laced types. Moreover, in this section, we observe that all $\gamma$-equivariant cluster tilting modules in $P(\bG_2)$ and $P(\bC_n)$ for $n \geq 3$ are admissible.

\begin{lemma} Let $T$ be a $\gamma$-equivariant cluster tilting module of $\La$, and set $E= \End_{\La}(T)$. Then the quiver of $E$ does not have
2-cycles in the part corresponding to non-projective summands of $T$.
\end{lemma}

\begin{proof}
Suppose we have arrows $\tilde{V}\to \tilde{Z}$ and $\tilde{Z}\to \tilde{V}$ in the quiver of $E$. Then by \cite[Proposition~3.11]{GLS-3}, for a simple module $S$ of $E$ we have
$\Ext^2(S, S) \neq 0$. In fact from the proof it follows that this is the case for $S=S_{\tilde{V}}$ or $S_{\tilde{Z}}$, without loss of generality assume this holds for $S=S_{\tilde{V}}$.
By Proposition \ref{P:Ext3-i} we obtain 
$0\neq \Ext^1(S_{\tilde{V}}, S_{_{\gamma^{-1}}\tilde{V}})$; a contradiction to Theorem \ref{T:maximal rigid = 2-ct}.
\end{proof}

\begin{corollary}
There is no path of the form $\tilde{V} \to \tilde{W} \to {}_\gamma \tilde{V}$ in the quiver of $E$ where at least one of $\tilde{V}, \tilde{W}$ is $\gamma$-equivariant. 
\end{corollary}

\begin{corollary}
	Any $\gamma$-equivariant cluster tilting module in $P(\bG_2)$ or $P(\bC_n)$ for $n \geq 3$ is admissible.
\end{corollary}

\section{An Example}
We describe in detail a mutation of the start module, when $\La = P(\bB_3)$. As the input, we take the category $\cC/\langle \tau\sigma\rangle$ as described in 
Section \ref{S:list} Item (2). Here $\cC = \mathbb{K}(\bZ\Delta)$ for $\Delta$ of type $A_5$ with the labelling as in 
Figure \ref{fig:standard labelling}. 
We take the Auslander category  to be the subcategory whose quiver is the connected subquiver with 15 vertices, with (unique) sink labelled by $0_0$.  

\subsection{The start module} 
The indecomposable summands of the start module $T$ of $P(\bB_3)$ are as follows. 

\begin{align*}
T(0_0) &= \tiny{ \begin{matrix} & & 0& & \cr
                                          &1&&2&\cr
                                          4&& 0&&3\cr
                                          &2&&1&\cr
                                          &&0&&\end{matrix}}, & T(0_1) &=\tiny{\begin{matrix} & & & & \cr
                                          &1&&2&\cr
                                          4&& 0&&3\cr
                                          &2&&1&\cr
                                          &&0&&\end{matrix}}, & T(0_2) &= \tiny{\begin{matrix}&&&&\cr&&&&\cr
                                          4&& &&3\cr
                                          &2&&1&\cr
                                          &&0&&\end{matrix}.}
\\
T(1_0) &= \tiny{\begin{matrix}&2&&&\cr
                                       3&&0&&\cr
                                       &1&&2\cr &&0&&3\cr&&&1&\end{matrix}}
                                       , &  T(1_1) &= \tiny{\begin{matrix}&&&&\cr
                                                                          3&&&&\cr
                                       &1&&2&\cr &&0&&3\cr&&&1&\end{matrix}}, &    T(1_2)& = \tiny{\begin{matrix}&&&&\cr
                                       &&&&\cr
                                       &&&&\cr &&&&3\cr&&&1&\end{matrix}} \\ 
T(3_0)& =\tiny{ \begin{matrix}4&&&&\cr &2&&&\cr&&0&&\cr&&&2&\cr&&&&3\end{matrix}}, & T(3_1) &= \tiny{\begin{matrix}2&\cr &3, & \end{matrix}} & T(3_2) & =  \tiny{3}. \\
T(2_i) & = {}_\gamma T(1_i) & T(4_i) &= {}_\gamma T(2_i) &  \text{for } i &= 1,2.
\end{align*}

Here the rows of a diagram describe the socle series. The module has a subquotient
with top $i$ and socle $j$ if and only if there is a subdiagram $\begin{matrix} i&\cr &j\end{matrix}$ or $\begin{matrix}&i\cr j&\end{matrix}$ in two neighbouring rows.

The module $T(x_i)$ is the push down of the injective module $I(x_i)$ of the Auslander category, viewed as a module over $\cC$.
Note that  each $T(x_i)$ is a submodule of $T(x_0)$ which is injective and projective as a  $\La$-module.

It is straightforward to compute the exchange sequences for the indecomposable summands of $T$, and hence the quiver of the endomorphism algebra $E= {\rm End}_{\La}(T)$, which is
as follows.

 
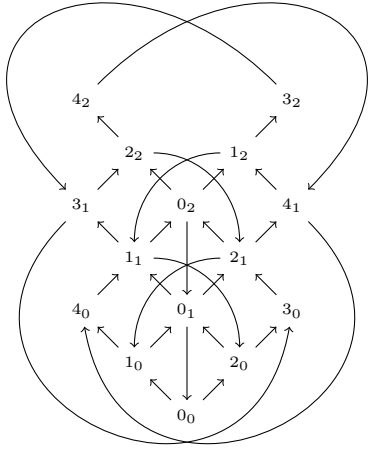
\begin{center}
\begin{tikzpicture}[scale = 0.7, font=\tiny]

%P(\bA_3)

%draw nodes
\node (32) at (-2,6) {$4_2$};
\node (42) at (2,6) {$3_2$};

\node (12) at (-1,5) {$2_2$};
\node (22) at (1,5) {$1_2$};

\node (31) at (-2,4) {$3_1$};
\node (02) at (0,4) {$0_2$};
\node (41) at (2,4) {$4_1$};

\node (11) at (-1,3) {$1_1$};
\node (21) at (1,3) {$2_1$};

\node (30) at (-2,2) {$4_0$};
\node (01) at (0,2) {$0_1$};
\node (40) at (2,2) {$3_0$};

\node (10) at (-1,1) {$1_0$};
\node (20) at (1,1) {$2_0$};

\node (00) at (0,0) {$0_0$};

%draw arrows AR quiver
\draw[<-] (32) -- (12);
\draw[<-] (42) -- (22);

\draw[<-] (12) -- (31);
\draw[<-] (12) -- (02);
\draw[<-] (22) -- (02);
\draw[<-] (22) -- (41);

\draw[<-] (31) -- (11);
\draw[<-] (02) -- (11);
\draw[<-] (02) -- (21);
\draw[<-] (41) -- (21);

\draw[<-] (11) -- (30);
\draw[<-] (11) -- (01);
\draw[<-] (21) -- (01);
\draw[<-] (21) -- (40);

\draw[<-] (30) -- (10);
\draw[<-] (01) -- (10);
\draw[<-] (01) -- (20);
\draw[<-] (40) -- (20);

\draw[<-] (10) -- (00);
\draw[<-] (20) -- (00);

%draw extra
\draw[<-] (00) -- (01);
\draw[<-] (01) -- (02);

\draw[<-] (10) to [out=90, in = 180] (21);
\draw[<-] (20) to [out=90, in = 0] (11);

\draw[<-] (11) to [out=90, in = 180] (22);
\draw[<-] (21) to [out=90, in = 0] (12);

\draw[<-] (40) to [out=260, in = 225, looseness=2.5] (31);
\draw[<-] (30) to [out=280, in = 315, looseness=2.5] (41);

\draw[<-] (41) to [out=45, in = 45, looseness=2.5] (32);
\draw[<-] (31) to [out=135, in = 135, looseness=2.5] (42);

\end{tikzpicture}

\end{center}

\subsection{The matrices}

The principal part $\tilde{B}_T^\circ$ of the adjacency matrix $\tilde{B}_T$ of the quiver of $E$, i.e.\ the submatrix of rows and columns labelled by non-projective indecomposable summands of $T$, is the matrix
\[ \tilde{B}_T^\circ =
	\tiny{\left(
	\begin{array}{c|cc|cc|c|cc|cc}
	0 & -1 & -1 & 0 & 0 & 1 & 0 & 0 & 0 & 0 \\
	\hline
	1 & 0 & 0 & -1& 0 & -1 & 1 & 0 & 0 & 0 \\
	1 & 0 & 0 & 0 & -1 & -1 & 0 & 1 & 0 & 0 \\
	\hline
	 0 & 1 & 0 & 0 & 0 & 0 & 0 & -1 & 1 & 0 \\
	 0 & 0 & 1 & 0 & 0 & 0 & -1 & 0 & 0 & 1 \\
	 \hline
	 -1 & 1 & 1 & 0 & 0 & 0 & -1 & -1 & 0 & 0 \\
	 \hline
	 0 & -1 & 0 & 0 & 1 & 1 & 0 & 0 & -1 & 0 \\
	 0 & 0 & -1 & 1 & 0 & 1 & 0 & 0 & 0 & -1 \\
	 \hline
	 0 & 0 & 0 & -1 & 0 & 0 & 1 & 0 & 0 & 0 \\
	 0 & 0 & 0 & 0 & -1 & 0 & 0 & 1 & 0 & 0 \\
	\end{array}
	\right)},
	%\end{array}
	%\end{pmatrix}.
\]
where the rows and columns are consecutively labelled by $0_1, 1_1, 2_1, 3_1, 4_1, 0_2, 1_2, 2_2, 3_2, 4_2$. Above, we have divided the matrix into blocks corresponding to the $\gamma$-orbits of the summands of $T$.

The principal part  $\tilde{R}_T^\circ$ is obtained from $\tilde{B}_T^\circ$ by replacing each diagonal $2\times 2$ block by $\left(\begin{matrix} 1 & -1\cr -1 & 1\end{matrix}\right)$, and by replacing each 
off-diagonal $2\times 2$ block with entries $\geq 0$ by its ``twist'', that is by interchanging
\[
	\left(\begin{matrix} 1 & 0\cr 0 & 1\end{matrix}\right) \ \ \mbox{ and } \ \ \left(\begin{matrix} 0&1\cr 1&0\end{matrix}\right).
\]
In particular one sees that the folded matrices $B_T^\circ = \tilde{B}_T^\circ/\langle \gamma\rangle$ and $R_T^\circ = \tilde{R}_T^\circ/\langle \gamma\rangle$ are equal. 
We have
\begin{align}\label{eqn:matrix}
	B^\circ_T = \left(\begin{matrix} 0 & -1 & 0 & 1 & 0 & 0\cr
2 & 0 & -1 & -2 & 1 & 0\cr
0 & 1 & 0 & 0 & -1 & 1\cr
-1 & 1 & 0 & 0 & -1 & 0\cr
0 & -1 & 1 & 2 & 0 & -1\cr
0 & 0 & -1 & 0 & 1 & 0\end{matrix}\right).
%\leqno{(1)}
\end{align}

\subsection{Mutation of the start module} We mutate in the direction of the minimal $\gamma$-equivariant summand $X = T(1_1) \oplus T(2_1)$, and set
\[
	\mu_X(T) = T^* = T/X \oplus Y.
\] 
%$[12]_1$. 
One computes the relevant exchange sequences:
\begin{align*}
T(1_1) \hookrightarrow T(1_0)\oplus T(3_1)\oplus T(0_2) \twoheadrightarrow Y(1_1) &\bigoplus   T(2_1) \hookrightarrow  T(2_0)\oplus T(4_1)\oplus T(0_2) \twoheadrightarrow Y(2_1)\\
Y(1_1) \hookrightarrow T(3_0)\oplus T(0_1)\oplus T(2_2) \twoheadrightarrow T(2_1) &\bigoplus Y(2_1) \hookrightarrow T(4_0)\oplus T(0_1)\oplus T(1_2) \twoheadrightarrow T(1_1)
\end{align*}
%\begin{align*}
%T(1_1) \hookrightarrow & T(1_0)\oplus T(3_1)\oplus T(0_2) \twoheadrightarrow Y(1_1),  &  Y(1_1) \hookrightarrow T(3_0)\oplus T(0_1)\oplus T(2_2) \twoheadrightarrow T(2_1)
%\end{align*}
%and their twists. 
We set $\tilde{X}:= T(1_1)$ and $\tilde{Y} = Y(1_1)$. 
In the proof of Proposition \ref{P:mutation Cartan}, we only need the first exchange sequence. With the notation from the proof of Proposition \ref{P:mutation Cartan} we have
\[
	\tilde{T}' = T(1_0)\oplus T(3_1)\oplus T(0_2) = \bigoplus_{\tilde{b}_{\tilde{Z} \tilde{X}}\geq 0}  \tilde{Z}^{\tilde{b}_{\tilde{Z}\tilde{X}}}.
\]
In Proposition \ref{P:mutation Cartan} and the proof of Theorem \ref{T:mutation matrix}, we have worked with matrices $U$ and $W$ such that $W^\circ B_T^\circ U^\circ = B^\circ_{T^*}$. 
Here, we have
\[
	U^\circ = \left(\begin{matrix} 1 & 0 & 0 & 0 & 0 & 0 \cr
						0 & -1 & 1 & 2 & 0 & 0 \cr
						0 & 0 & 1 & 0 & 0 & 0 \cr
						0 & 0 & 0 & 1 & 0 & 0 \cr
						0 & 0 & 0 & 0 & 1 & 0 \cr
						0 & 0 & 0 & 0 & 0 & 1
\end{matrix}\right)
 \ \ \ \ \text{and} \ \ \ \ 
 W^\circ = \left(\begin{matrix} 1 & 0 & 0 & 0 & 0 & 0 \cr
						0 & -1 & 0 & 0 & 0 & 0 \cr
						0 & 1 & 1 & 0 & 0 & 0 \cr
						0 & 1 & 0 & 1 & 0 & 0 \cr
						0 & 0 & 0 & 0 & 1 & 0 \cr
						0 & 0 & 0 & 0 & 0 & 1
\end{matrix}\right).
\]

\subsection{Matrix mutation} We compare the above with the matrix mutation. We apply Lemma \ref{L:matrix mutation} with $A= B^\circ$,the matrix from Equation (\ref{eqn:matrix}) and mutate at $k=2$, i.e.\ the row labelled by the minimal $\gamma$-equivariant summand  $X = T(1_1) \oplus T(2_1)$ of $T$. 
Then the matrix $U_{A, k}$ is $U^\circ$, and
the matrix $V_{A, k}^t$ is the matrix $W^\circ$. We see that
\[
	\mu_2(B_T^\circ) = W^\circ B^\circ U^\circ = \left(\begin{matrix} 0 & 1 & -1 & -1 & 0 & 0 \cr
						-2 & 0 & 1 & 2 & -1 & 0 \cr
						2 & -1 & 0 & 0 & 0 & 1 \cr
						1 & -1 & 0 & 0 & 0 & 0 \cr
						0 & 1 & 0 & 0 & 0 & -1 \cr
						0 & 0 & -1 & 0 & 1 & 0
\end{matrix}\right) = B_{T^*}^\circ.
\]

\bibliographystyle{abbrv}

\bibliography{sira_bib}

\end{document}